\documentclass[a4paper, 10pt]{amsart}
\usepackage{latexsym}
\usepackage{tikz-cd}
\usepackage{amssymb}
\usepackage{xcolor}
\usepackage[normalem]{ulem}

\newtheorem*{KN}{Crazy Knight's Tour Problem}

\def\Z{\mathbb{Z}}

\newcommand{\probname}{Crazy Knight's Tour Problem}

\def\G{\Gamma}

\def\H{\mathrm{H}}

\def\E{\mathcal{E}}

\def\R{\mathcal{R}}
\def\C{\mathcal{C}}

\usepackage{mathtools}

\usepackage{geometry}
\newtheorem{thm}{Theorem}[section]
\newtheorem{lem}[thm]{Lemma}
\newtheorem{cor}[thm]{Corollary}
\newtheorem{prop}[thm]{Proposition}

\newtheorem{ex}[thm]{Example}
\newtheorem{rem}[thm]{Remark}

\theoremstyle{definition}
\newtheorem{defin}[thm]{Definition}
\numberwithin{equation}{section}

\def\Z{\mathbb{Z}}

\title[Non-isomorphic simple $k$-gonal biembeddings]{On the number of non-isomorphic (simple) $k$-gonal biembeddings of complete multipartite graphs}
\author{Simone Costa}
\address{DICATAM, Universit\`a degli Studi di Brescia, Via
Branze 43, 25123 Brescia, Italy}
\email{simone.costa@unibs.it}

\author{Anita Pasotti}
\address{DICATAM, Universit\`a degli Studi di Brescia, Via
Branze 43, 25123 Brescia, Italy}
\email{anita.pasotti@unibs.it}


\begin{document}
\maketitle
\begin{abstract}
This article aims to provide exponential lower bounds on the number of non-isomorphic $k$-gonal biembeddings of the complete multipartite graph into orientable surfaces.

For this purpose, we use the concept, introduced by Archdeacon in 2015, of Heffer array and its relations with graph embeddings.
In particular we show that, under certain hypotheses, from a single Heffter array, we can obtain an exponential number of distinct graph embeddings.
Exploiting this idea starting from the arrays constructed by Cavenagh, Donovan and Yaz\i c\i \ in 2020, we obtain that, for infinitely many values of $k$ and $v$, there are at least $k^{\frac{k}{2}+o(k)} \cdot 2^{v\cdot \frac{H(1/4)}{(2k)^2}+o(v)}$ non-isomorphic $k$-gonal biembeddings of $K_v$, where $H(\cdot)$ is the binary entropy.
Moreover about the embeddings of $K_{\frac{v}{t}\times t}$, for $t\in\{1,2,k\}$, we provide a construction of $2^{v\cdot \frac{H(1/4)}{2k(k-1)}+o(v,k)}$ non-isomorphic $k$-gonal biembeddings whenever $k$ is odd and $v$ belongs to a wide infinite family of values.
\end{abstract}

\section{Introduction}
The purpose of this paper is to provide exponential lower bounds on the number of non-isomorphic embeddings of
the complete multipartite graph  into orientable surfaces that induce faces of a given length $k$ (i.e. we investigate the so-called \emph{$k$-gonal embeddings}).
We first recall some basic definitions, see \cite{Moh}.
\begin{defin}
Given a graph $\G$ and a surface $\Sigma$, an \emph{embedding} of $\G$ in $\Sigma$ is a continuous injective mapping $\psi: \G \rightarrow \Sigma$, where $\G$ is viewed with the usual topology as $1$-dimensional simplicial complex.
\end{defin}
The connected components of $\Sigma \setminus \psi(\G)$ are said $\psi$-\emph{faces}. Also, with abuse of notation, we say that a circuit $F$ of $\G$ is a face (induced by the embedding $\psi$) if $\psi(F)$ is the boundary of a $\psi$-face. Then, if each $\psi$-face is homeomorphic to an open disc, the embedding $\psi$ is called \emph{cellular}.
If the boundary of a face is homeomorphic to a circumference, such a face is said \emph{simple} and if all the faces are simple we say that the embedding is simple.
 If moreover, the embedding is face $2$-colorable, we say that it is a \emph{biembedding}.
In this context, we say that two embeddings $\psi: \Gamma \rightarrow \Sigma$ and $\psi': \Gamma' \rightarrow \Sigma'$ are \emph{isomorphic} if and only if there is a graph isomorphism $\sigma: \Gamma\rightarrow \Gamma'$ such that $\sigma(F)$ is a $\psi'$-face if and only if $F$ is a $\psi$-face.

The existence problem of cellular embeddings of a graph $\Gamma$ into (orientable) surfaces has been widely studied in the case of triangular embeddings, which are the ones whose faces are triangular. This kind of embeddings has been investigated, at first, because their construction was a major step in proving the Map Color Theorem \cite{R}. Among the papers related to this existence problem, we recall \cite{Bonnington, GG08, GGS, GK10A, GK10B, Korzhik21, LNW} where the natural question of the rate growth of the number of non-isomorphic triangular embeddings of complete graphs has been considered too.
Moreover, due to the Euler formula, if there exists a triangular embedding $\psi$ from $\Gamma$ to some surface $\Sigma$, $\psi$ minimizes the genus of $\Sigma$. For this reason, such kinds of embeddings are called \emph{genus embeddings}.
Two naturally related questions are the investigation of the rate of the number of non-isomorphic genus embeddings (see \cite{Korzhik10, Korzhik2}) and that of the $k$-gonal embeddings (see \cite{GGSHamiltonian, Korzhik12,Korzhik}).

In this paper, we consider the latter question and we study the rate growth of the number of non-isomorphic $k$-gonal embeddings of the complete multipartite graph with $m$ parts of size $t$, denoted by $K_{m\times t}$.
Here, we  provide exponential lower bounds on this number for several infinite classes of parameters $k,m$ and $t$.
Furthermore, our embeddings  also realize additional properties: the faces they induce are (in several cases) simple
and it is possible to color them within two colors, i.e. these embeddings are $2$-face colorable.
Finally, in the cases where $k$ is $3$, we find new classes of genus embeddings.

The approach we use in this article is purely combinatorial and requires the notion of combinatorial embedding, see \cite{GG, JS}.
Here, we denote by $D(\Gamma)$ the set of all the oriented edges of the graph $\Gamma$ and,
given a vertex $x$ of $\G$,
 by $N(\Gamma,x)$ the neighborhood of $x$ in $\Gamma$.
\begin{defin}
Let $\Gamma$ be a connected multigraph. A \emph{combinatorial embedding} of $\Gamma$ (into an orientable surface) is a pair $\Pi=(\Gamma,\rho)$ where $\rho: D(\Gamma)\rightarrow D(\Gamma)$ satisfies the following properties:
\begin{itemize}
\item for any $y\in N(\Gamma,x)$, there exists $y'\in N(\Gamma,x)$ such that $\rho(x,y)=(x,y')$;
\item we define $\rho_x$ as the permutation of $N(\Gamma,x)$ such that, given $y\in N(\Gamma,x)$, $\rho(x,y)=(x,\rho_x(y))$. Then the permutation $\rho_x$ is a cycle of order $|N(\Gamma,x)|$.
\end{itemize}
\end{defin}
It is well known that a combinatorial embedding of $\Gamma$ is equivalent to a cellular embedding of $\Gamma$ in an orientable surface, see \cite{A,GT, MT}. This observation leads us to study this kind of embedding isomorphisms purely combinatorially. From the combinatorial point of view, the faces are determined using the face-trace algorithm, see \cite{A}.
It is easy to see that the faces are circuits (that is sequences of consecutive vertices and edges, denoted by $v_1,v_2,\ldots,v_k$), the \emph{length} of a circuit is the number of its edges. If the faces are simple then $v_i\neq v_j$ for any $i\neq j$, so the circuits
are indeed cycles with $k$ distinct vertices and $k$ edges.
In this context, it is possible to rephrase the definition of embedding isomorphism as done by
Korzhik and Voss in \cite{Korzhik}, see page 61.
\begin{defin}\label{DefEmbeddings}
Let $\Pi:= (\Gamma,\rho)$ and $\Pi':= (\Gamma',\rho')$ be two combinatorial embeddings of, respectively, $\Gamma$ and $\Gamma'$. We say that $\Pi$ is \emph{isomorphic} to $\Pi'$ if there exists a graph isomorphism $\sigma: \Gamma\rightarrow \Gamma'$ such that, for any $(x,y)\in D(\Gamma)$, we have either
\begin{equation}\label{eq11}
\sigma\circ \rho(x,y)=\rho'\circ \sigma(x,y)
\end{equation}
or
\begin{equation}\label{eq12}
\sigma\circ \rho(x,y)=(\rho')^{-1}\circ \sigma(x,y).
\end{equation}
We also say, with abuse of notation, that $\sigma$ is an \emph{embedding isomorphism} between $\Pi$ and $\Pi'$.
Moreover, if equation (\ref{eq11}) holds, $\sigma$ is said to be an \emph{orientation preserving isomorphism} while,
if (\ref{eq12}) holds, $\sigma$ is said to be an \emph{orientation reversing isomorphism}.
\end{defin}

This combinatorial approach has been developed in the literature into two kinds of directions. The first one is the use of recursive constructions and has been applied to construct triangular embeddings of complete graphs from triangular embeddings of complete graphs of a lesser order. Within this method, it was first shown that there are at least $2^{av^2-o(v^2)}$ non-isomorphic face $2$-colorable triangular embeddings of the complete graph $K_v$ for several congruence classes modulo $36$, $60$ and $84$ (see \cite{Bonnington, GGS}) and then that, for an infinite (but rather sparse) family of values of $v$, there are at least $v^{bv^2-o(v^2)}$ non-isomorphic face $2$-colorable triangular embeddings of $K_v$ (see \cite{GG08,GK10A,GK10B}). Another consequence of these kinds of recursive constructions is the existence of $2^{cv^2-o(v^2)}$ non-isomorphic Hamiltonian embeddings of $K_v$ for infinitely many values of $v$ (see \cite{GGSHamiltonian}).

The second approach uses the current graph technique. Within this method, it was provided the first exponential lower bound (of type $2^{dv}$) on the number of non-isomorphic face $2$-colorable triangular embeddings of $K_v$ for infinitely many values of $v$. Then, similar results have been also given in the cases of genus and quadrangular embeddings (see \cite{Korzhik12,Korzhik21,Korzhik,Korzhik2}).
The approach used in this paper belongs to this second family.
The main tool we will use is the concept of Heffter array, introduced by Archdeacon in \cite{A} to provide constructions of current graphs.
Section 2 of this paper will be dedicated to introducing this kind of array, to reviewing the literature on this topic
and to further investigating the connection with biembeddings.
Then, in Section 3, we will deal with the following problem: given a family of embeddings each of which admits $\mathbb{Z}_v$ as a regular automorphism group (i.e. embeddings that are $\mathbb{Z}_v$-regular), how many of its elements can be isomorphic? Proposition \ref{upperboundFamily} will provide an upper bound on this number.
In the last two sections, we will consider some of the known constructions of Heffter arrays and we will show that, under certain hypotheses, from each of such arrays we can obtain a family of $\mathbb{Z}_v$-regular embeddings that is exponentially big. 
These families, together with Proposition \ref{upperboundFamily}, will allow us to achieve the existence of an exponential number of non-isomorphic $k$-gonal biembeddings of $K_v$ and $K_{\frac{v}{t}\times t}$ in several situations.
In particular, in Section 4 we will obtain that, when $k$ is congruent to $3$ modulo $4$ and $v$ belongs to an infinite family of values, there are $k^{\frac{k}{2}+o(k)} \cdot 2^{g(k)v+o(v)}$ non-isomorphic $k$-gonal biembeddings of $K_v$ where $g(k)$ is a rational function of $k$. Finally, in Section 5, we will consider the embeddings of $K_{\frac{v}{t}\times t}$. In this case, for $t\in\{1,2,k\}$, we will provide a construction of $2^{h(k)v+o(v,k)}$ non-isomorphic $k$-gonal biembeddings whenever $k$ is odd, $v$ belongs to a wide infinite family of values and where $h(k)$ is a rational function of $k$.

\section{Heffter arrays and biembeddings}
In this section we introduce the classical concept of Heffter array and its generalizations,
showing how these notions are useful tools for getting biembeddings of the complete
multipartite graph into an orientable surface.

An $m\times n$ partially filled (p.f., for short) array on a given set $\Omega$ is an $m\times n$ matrix with elements in $\Omega$
in which some cells can be empty. Archdeacon \cite{A} introduced a class of p.f. arrays, called \emph{Heffter arrays}, and showed how it is related to several other mathematical concepts such as difference families, graph decompositions, current graphs and biembeddings.
These arrays have been then generalized by Costa and al. in \cite{RelH}  as follows.
\begin{defin}\label{def:Heffter}
  Let $v=2nk+t$ be a positive integer, where $t$ divides $2nk$, and let $J$ be the subgroup of $\Z_v$ of order $t$.
  A \emph{Heffter array over $\Z_v$ relative to $J$}, denoted by $\H_t(m,n;h,k)$, is an $m\times n$ p.f. array with elements in
  $\Z_v$ such that:
  \begin{itemize}
  \item [(1)] each row contains $h$ filled cells and each column contains $k$ filled cells;
  \item [(2)] for every $x\in \Z_v \setminus J$, either $x$ or $-x$ appears in the array;
  \item [(3)] the elements in every row and in every column sum to $0$ (in $\Z_{v}$).
  \end{itemize}
\end{defin}

If $t=1$, namely, if $J$ is the trivial subgroup of $\Z_{2nk+1}$, we find the classical Heffter arrays defined by Archdeacon,
which are simply denoted by  $\H(m,n;h,k)$.
It is immediate that if there exists a $\H_t(m,n;h,k)$ then $mh=nk$, $3\leq h \leq n$ and $3 \leq k \leq m$. Also, $m=n$ implies $k=h$
and a $\H_t(n,n;k,k)$ is simply denoted by $\H_t(n;k)$. The most important result about the existence problem for Heffter arrays is the following, see \cite{ADDY, CDDY, DW}.
\begin{thm}
A $\H(n;k)$ exists for every $n \geq k \geq 3$.
\end{thm}
For other existence results see \cite{MP3} and the references therein.

In \cite{CPEJC} we introduced the further generalization of a \emph{$\lambda$-fold Heffter array $A$ over $\Z_v$ relative to $J$},
denoted by $^\lambda \H_t(m,n;h,k)$ replacing property (2) of Definition \ref{def:Heffter} with the following one:
  \begin{itemize}
  \item [(2')] the multiset $\{\pm x \mid x\in A\}$ contains $\lambda$ times each element of $\Z_v\setminus J$,
  where $v=\frac{2nk}{\lambda}+t$.
  \end{itemize}

 Note that if $\lambda>1$ then $h$ and $k$ can also be equal to $2$.
 \begin{ex}\label{ex>1}
The following array $A$ is a $^2 \H(2,5;5,2)$, in fact the  multiset $\{\pm x \mid x\in A\}$ contains
$2$ times each element of $\Z_{11}\setminus \{0\}$.
$$\begin{array}{|r|r|r|r|r|}
\hline
  1 &  -2 & 3 & 4 & 5  \\ \hline
   -1 &  2 & -3 & -4 & -5  \\ \hline
\end{array}$$
\end{ex}

  Anyway in this paper we focus on the case $\lambda=1$, since several of our constructions cannot be naturally
  extended to the case $\lambda >1$, as it will be underlined in Remark \ref{NOlambda}.

\begin{ex}\label{ex}
Below we have a $\H_9(11;9)$, say $A$. Hence the elements of $A$ belongs to $\Z_{207}$
and we avoid the elements of the subgroup of $\Z_{207}$ of order $9$.
$$\begin{array}{|r|r|r|r|r|r|r|r|r|r|r|}
\hline
  10 & 55 & 101 & -90 &  & 13 & -22 &  & -78 & 67 & -56 \\ \hline
  -37 & -9 & 45 & 102 & -91 &  & 21 & -20 &  & -79 & 68 \\ \hline
  58 & -47 & -8 & 54 & 103 & -81 &  & 19 & -18 &  & -80 \\ \hline
  -70 & 59 & -38 & -7 & 44 & 93 & -82 &  & 17 & -16 &  \\ \hline
   & -71 & 60 & -48 & -6 & 53 & 94 & -83 &  & 15 & -14 \\ \hline
  -33 &  & -72 & 61 & -39 & 11 & 49 & 95 & -84 &  & 12 \\ \hline
  24 & -25 &  & -73 & 62 & -43 & 4 & 40 & 96 & -85 &  \\ \hline
   & 26 & -27 &  & -74 & 63 & -52 & 3 & 50 & 97 & -86 \\ \hline
  -87 &  & 28 & -29 &  & -75 & 64 & -42 & 2 & 41 & 98 \\ \hline
  99 & -88 &  & 30 & -31 &  & -76 & 65 & -51 & -5 & 57 \\\hline
  36 & 100 & -89 & & 32 & -34 &  & -77 & 66 & -35  & 1\\ \hline
\end{array}$$
\end{ex}

The focus of this paper is not the existence problem of Heffter arrays, but their connection with face $2$-colorable embeddings.
We point out that there are several papers in which Heffter arrays have been investigated to obtain
 biembeddings see \cite{A,BCDY, CDDYbiem, CMPPHeffter, CPPBiembeddings, CPEJC,DM}.
To present such a  connection, now we have to introduce the concepts of \emph{simple and compatible orderings}.

In the following, given two integers $a\leq b$, by $[a,b]$ we  denote the interval containing the integers $\{a,a+1,\ldots,b\}$.
If $a> b$, then $[a,b]$ is empty. The rows and the columns of an $m\times n$ array $A$ are denoted by $R_1,\ldots, R_m$ and by $C_1,\ldots, C_n$, respectively. Also we denote by $\E(A)$, $\E(R_i)$, $\E(C_j)$ the list of the elements of the filled cells of $A$, of the $i$-th row and of the $j$-th column, respectively. Given a finite subset $T$ of an abelian group $G$ and an ordering $\omega=(t_1, t_2,\ldots, t_k)$ of the elements of $T$, for any $i\in [1,k]$ let $s_i=\sum_{j=1}^{i}t_j$ be the $i$-th partial sum of $T$. The ordering $\omega$ is said to be \emph{simple} if $s_a\neq s_b$ for all $1\leq a < b \leq k$. We point out that if $s_k=0$ an ordering $\omega$ is simple if no proper subsequence of consecutive elements of $\omega$ sums to $0$. Note also that, if $\omega$ is a simple ordering, then $\omega^{-1}=(t_k,t_{k-1},\ldots, t_1)$ is simple too. Given an $m\times n$ p.f. array $A$, by $\omega_{R_i}$ and $\omega_{C_j}$ we denote an ordering of $\E(R_i)$ and $\E(C_j)$, respectively. If for any $i\in [1,m]$ and for any $j\in [1,n]$, the orderings $\omega_{R_i}$
and $\omega_{C_j}$ are simple, we define by $\omega_r=\omega_{R_1}\circ \cdots \circ \omega_{R_m}$ the simple ordering for the rows and by $\omega_c=\omega_{C_1}\circ \cdots \circ \omega_{C_n}$ the simple ordering for the columns. Also, by \emph{natural ordering} of a row (column) of $A$ one means the ordering from left to right (from top to bottom).
\begin{defin}
  A p.f. array $A$ on an abelian group $G$ is said to be
  \begin{itemize}
    \item \emph{simple} if there exists a simple ordering for each row and each column of $A$;
    \item \emph{globally simple} if the natural ordering of each row and each column of $A$ is simple.
  \end{itemize}
\end{defin}
It is easy to see that if $k\leq 4$ then every $\H_t(n;k)$ is globally simple. By a direct check one can see that the array of
Example \ref{ex} is globally simple.

\begin{defin}
  Given a relative Heffter array $A$, the orderings $\omega_r$ and $\omega_c$ are said to be \emph{compatible} if $\omega_c \circ \omega_r$ is a cycle of order $|\E(A)|$.
\end{defin}
 Reasoning as in \cite{CPPBiembeddings}, we get the following.
\begin{thm}\label{HeffterBiemb}
 Let $A$ be a relative Heffter array $\H_t(m,n;h,k)$ that admits two compatible orderings $\omega_r$ and $\omega_c$. Then there exists a cellular biembedding $\sigma$ of
  $K_{\frac{2nk+t}{t}\times t}$, such
   that every edge is on a face whose boundary has length
  $h$ and on a face whose boundary has length $k$,
  into an orientable surface of genus
  $$g=1+\frac{(nk-n-m-1)(2nk+t)}{2}.$$
  Moreover, setting $v=2nk+t$, $\sigma$ is $\Z_v$-regular.
\end{thm}

\begin{rem}
As already remarked in the Introduction, in general, in Theorem \ref{HeffterBiemb}, the faces are circuits, but if the array is simple
with respect to the compatible orderings $\omega_r$ and $\omega_c$ then the faces are cycles.
Clearly, in this case the biembedding is simple.
\end{rem}

Now we recall the definition of the Archdeacon embedding, see \cite{A}. Let $A$ be a $\H_t(m,n;h,k)$; we consider the permutation $\rho_0$ on
$\pm \E(A)=\Z_{2nk+t}\setminus \frac{2nk+t}{t}\Z_{2nk+t}$, where $\frac{2nk+t}{t}\Z_{2nk+t}$
denotes the subgroup of $\Z_{2nk+t}$ of order $t$, so defined:
\begin{eqnarray}\label{ArchEmb}
\rho_0(a)&=&\begin{cases}
-\omega_r(a)\mbox{ if } a\in \E(A);\\
\omega_c(-a)\mbox{ if } a\in -\E(A).\\
\end{cases}
\end{eqnarray}
Note that the complete multipartite graph $K_{\frac{2nk+t}{t}\times t}$ is nothing but the Cayley graph on $\Z_{2nk+t}$
with connection set $\pm \E(A)$, denoted by $Cay[\Z_{2nk+t} : \pm \E(A)]$.
Now, we define a map $\rho$ on the set of the oriented edges of this graph as follows:
\begin{eqnarray}\label{ArchRho}
\rho((x,x+a))&=& (x,x+\rho_0(a)).
\end{eqnarray}
Since $\rho_0$ acts cyclically on $\pm \E(A)$, the map $\rho$ is a rotation of $Cay[\Z_{2nk+t} : \pm \E(A)]$.

\begin{ex}\label{orderings}
  Let $A$ be the $\H_9(11;9)$ given in Example \ref{ex}. Consider the following ordering for the rows
  $$\small{\begin{array}{rcl}
  \omega_r&=&( 10,55,101,-90,13,-22, -78, 67, -56 )( -37,-9, 45,102,-91,21,-20,-79,68)\\
  & & ( 58,-47, -8, 54, 103,-81, 19,-18,-80)(-70,59,-38,-7,44,93,-82,17,-16)\\
  & & (-71,60,-48,-6,53,94,-83,15,-14 )(-33,-72,61,-39,11,49,95,-84,12 )\\
  & & (24,-25,-73,62,-43,4,40,96,-85)(26,-27,-74,63,-52,3,50,97,-86)\\
  & & (-87,28,-29,-75,64,-42, 2, 41,98)(99,-88,30,-31,-76,65,-51,-5,57)\\
  & & (36,100,-89, 32,-34, -77, 66,-35,1)
    \end{array}}$$
    and the following ordering for the columns
      $$\small{\begin{array}{rcl}
  \omega_c&=& (10,36,99,-87,24,-33,-70,58,-37)(55,-9,-47,59,-71,-25,26,-88,100)\\
   & & (101, 45, -8,-38,60,-72,-27,28,-89)(-90,102,54,-7,-48,61,-73,-29,30)\\
   & & (-91,103,44,-6,-39,62,-74,-31,32)(13,-81,93,53,11,-43,63,-75,-34)\\
   & & (-22, 21, -82,94,49,4,-52,64,-76)(-20,19,-83,95,40,3,-42,65,-77)\\
   & & (-78,-18,17,-84,96,50,2,-51,66)(67,-79,-16,15,-85,97,41,-5,-35)\\
   & & (-56,68,-80,-14,12,-86,98,57,1).
  \end{array}}$$
  Hence,
  $$\small{\begin{array}{rcl}
  \omega_c \circ \omega_r&=& (10,-9,-8,-7,-6,11,4,3,2,-5,1,99,100,101,102,103,93,94,95,96,97,\\
  & & 98, 24,26,28,30,32, 13,21,19,17,15,12,-70,-71,-72,-73,-74,-75,\\
  & & -76,-77,-78,-79,-80,-37,-47,-38,-48,-39,-43,-52,-42,-51,\\
  & & -35,-56,36,55,45,54,44,53,49,40,50,41,57,-87,-89,-91,-82,-84,\\
  & & -86,-88,-90,-81,-83,-85,-33,-27,-31,-22,-18,-14,-25,-29,\\
  & & -34,-20,-16,58,59,60,61,62,63,64,65,66,67,68).
  \end{array}}$$
  So, since $\omega_c \circ \omega_r$ is a cycle of order $99=|\E(A)|$ the orderings are compatible.

\end{ex}

Looking for compatible orderings in the case of a globally simple Heffter array leads us to consider the following problem
introduced in \cite{CDP}.
Given an $m\times n$ \emph{toroidal} p.f. array $A$, by $r_i$ we denote the orientation of the $i$-th row,
precisely $r_i=1$ if it is from left to right and $r_i=-1$ if it is from right to left. Analogously, for the $j$-th
column, if its orientation $c_j$ is from top to bottom then $c_j=1$ otherwise $c_j=-1$. Assume that an orientation
$\R=(r_1,\dots,r_m)$
and $\C=(c_1,\dots,c_n)$ is fixed. Given an initial filled cell $(i_1,j_1)$ consider the sequence
$ L_{\R,\C}(i_1,j_1)=((i_1,j_1),(i_2,j_2),\ldots,(i_\ell,j_\ell),$ $(i_{\ell+1},j_{\ell+1}),\ldots)$
where $j_{\ell+1}$ is the column index of the filled cell $(i_\ell,j_{\ell+1})$ of the row $R_{i_\ell}$ next
to
$(i_\ell,j_\ell)$ in the orientation $r_{i_\ell}$,
and where $i_{\ell+1}$ is the row index of the filled cell of the column $C_{j_{\ell+1}}$ next to
$(i_\ell,j_{\ell+1})$ in the orientation $c_{j_{\ell+1}}$.
Given an element $(i_k,j_k)\in L_{\R,\C}(i_1,j_1)$ we define $S_{\R,\C}(i_k,j_k)$ as the element
$(i_{k+1},j_{k+1})\in L_{\R,\C}(i_1,j_1)$. It is easy to see that $S_{\R,\C}$ is well defined on the set of the filled cells of $A$.

The problem proposed in \cite{CDP} is the following:
\begin{KN}
Given a toroidal p.f. array $A$,
do there exist $\R$ and $\C$ such that the list $L_{\R,\C}$ covers all the filled
cells of $A$?
\end{KN}

The \probname\ for a given array $A$ is denoted by $P(A)$, known results can be found in \cite{CDP}.
Also, given a filled cell $(i,j)$, if $L_{\R,\C}(i,j)$ covers all the filled positions of $A$ we will
say that $(\R,\C)$ is a solution of $P(A)$.
The relationship between the Crazy Knight's Tour Problem and (globally simple) relative Heffter arrays is explained in the following result,
see \cite{CPPBiembeddings}.

\begin{cor}\label{preprecedente}
Let $A$ be a  relative Heffter array $\H_t(m,n;h,k)$ such that $P(A)$ admits a solution $(\R,\C)$. Then there exists a biembedding of $K_{\frac{2nk+t}{t}\times t}$, such that every edge is on a face whose boundary has length
  $h$ and on a face whose boundary has length $k$,
 into an orientable surface.

Moreover if $A$ is globally simple, then the biembedding is simple.
\end{cor}

\begin{ex}\label{exPA}
Let $A$ be the $H_9(11;9)$ of Example \ref{ex}. Let $\R=(1,1,\ldots,1)$ and $\C=(-1,1,1,\ldots,1)$.
Now we consider  $S_{\R,\C}(1,1)$
and, in the following table, in each position we write $j$ if we reach that position after having applied $S_{\R,\C}$ to $(1,1)$
exactly $j$ times.
$$\begin{array}{r|r|r|r|r|r|r|r|r|r|r|r|}
& \uparrow & \downarrow & \downarrow & \downarrow & \downarrow & \downarrow & \downarrow & \downarrow & \downarrow & \downarrow & \downarrow\\
\hline
 \rightarrow & 0 & 56 & 13 & 73 &  & 27 & 80 &  & 41 & 97 & 54 \\ \hline
 \rightarrow &  44 & 1 & 57 & 14 & 68 &  & 28 & 86 &  & 42 & 98 \\ \hline
  \rightarrow & 88 & 45 & 2 & 58 & 15 & 74 &  & 29 & 81 &  & 43 \\ \hline
 \rightarrow &  33 & 89 & 46 & 3 & 59 & 16 & 69 &  & 30 & 87 &  \\ \hline
   \rightarrow & & 34 & 90 & 47 & 4 & 60 & 17 & 75 &  & 31 & 82 \\ \hline
  \rightarrow & 77 &  & 35 & 91 & 48 & 5 & 61 & 18 & 70 &  & 32 \\ \hline
 \rightarrow &  22 & 83 &  & 36 & 92 & 49 & 6 & 62 & 19 & 76 &  \\ \hline
   \rightarrow & & 23 & 78 &  & 37 & 93 & 50 & 7 & 63 & 20 & 71 \\ \hline
 \rightarrow &  66 &  & 24 & 84 &  & 38 & 94 & 51 & 8 & 64 & 21 \\ \hline
  \rightarrow & 11 & 72 &  & 25 & 79 &  & 39 & 95 & 52 & 9 & 65 \\\hline
  \rightarrow & 55 & 12 & 67 & & 26 & 85 &  & 40 & 96 & 53  & 10\\ \hline
\end{array}$$
Note that $L_{\R,\C}(1,1)$ covers all filled cells of $A$, hence $(\R,\C)$ is a solution of $P(A)$.
\end{ex}

\begin{rem}\label{RCordinering}
The orderings $\omega_r$ and $\omega_c$ of $A$ described in Example \ref{orderings},
correspond to the vectors $\R$ and $\C$ of Example \ref{exPA}, respectively.
Hence,  we also say that $(\R,\C)$ induces the cycle $\omega_c \circ \omega_r$.

Clearly, given an array $A$, a pair $(\R,\C)$ is a solution of $P(A)$ if and only if  the induced permutation
$\omega_c \circ \omega_r$ is a cycle of order $|\E(A)|$.
\end{rem}

Now, to present the results of this section we need some other definitions and notations.
By  $skel(A)$ we denote the \emph{skeleton} of $A$, that is the set of the filled positions of $A$.
Given an $n \times n$ p.f. array $A$, for $i\in [1,n]$ we define the $i$-th diagonal of $A$ as follows:
$$D_i=\{(i,1),(i+1,2),\ldots,(i-1,n)\}.$$
Here all the arithmetic on the row and column indices is performed modulo $n$, where $\{1,2,\ldots,n\}$ is the set of reduced residues.
The diagonals $D_{i+1}, D_{i+2}, \ldots, D_{i+k}$ are called $k$ \emph{consecutive diagonals}.
 A set of $t$ consecutive diagonals $S=\{D_{i+1},D_{i+2},\ldots, D_{i+t}\}$
is said to be an \emph{empty strip of width $t$} if $D_{i+1},D_{i+2},\ldots, D_{i+t}$ are empty diagonals, while $D_i$
and $D_{i+t+1}$ are non-empty diagonals.
\begin{defin}
  Let $n,k$ be integers such that $n\geq k\geq 1$. An $n\times n$ p.f. array $A$ is said to be:
  \begin{itemize}
    \item \emph{$k$-diagonal} if the non-empty cells of $A$ are exactly those of $k$ diagonals;
    \item \emph{cyclically $k$-diagonal} if the non-empty cells of $A$ are exactly those of $k$ consecutive diagonals;
     \item \emph{$k$-diagonal with width $t_1,t_2,\ldots, t_s$} if it is $k$-diagonal and has $s$ empty strips with width $t_1,t_2,\ldots,t_s$, respectively;
      \item \emph{$k$-diagonal with width $t$} if it is $k$-diagonal and all its empty strips have width $t$.
  \end{itemize}
\end{defin}
Clearly a cyclically $k$-diagonal array of size $n$ is nothing but a $k$-diagonal array with width $n-k$.
Note that the array of Example \ref{ex} is $9$-diagonal with width $1$.

\begin{lem}\label{-R-C}
 Let $A$ be a p.f. array. If $(\R,\C)$ is a solution of $P(A)$, then also $(-\R,-\C)$ is
 a solution of $P(A)$.
\end{lem}
\begin{proof}
By Remark \ref{RCordinering}, if $(\R,\C)$ is a solution of $P(A)$, then the induced cycle $\omega_c \circ \omega_r$ has order $|\E(A)|$.
Clearly also $(\omega_c \circ \omega_r)^{-1}= \omega_r^{-1} \circ \omega_c^{-1}$ is a cycle of the same order.
The same holds if we consider the conjugate  $\omega_r \circ (\omega_r^{-1} \circ \omega_c^{-1}) \circ \omega_r^{-1}=
\omega_c^{-1} \circ \omega_r^{-1}$, hence $(-\R,-\C)$ is a solution, too.
\end{proof}

\begin{lem}\label{CR}
 Let $A$ be a cyclically $k$-diagonal array of size $n\geq k$ and let $\R=(1,1,\ldots,$ $1)$. If $(\R,\C)$ is a solution of $P(A)$, then also $(\C,\R)$ is
 a solution of $P(A)$.
\end{lem}
\begin{proof}
We can assume, without loss of generality, that $(1,1)$ is a filled cell of $A$. If $(\R,\C)$ is a solution of $P(A)$, then the induced cycle $\omega_c \circ \omega_r$ has order $|\E(A)|$. Now,
since if we commute $\omega_r$ and $\omega_c$ we still obtain a cycle of order $|\E(A)|$,
 then $(\C,\R)$ is a solution of $P(A^t)$, where by $A^t$ we denote the transposed of $A$.
Note that, in general, $A$ and $A^t$ do not have the same skeleton.
For example, below $A$ is a cyclically $4$-diagonal array of size $6$ (we put a ``$\bullet$'' in the filled cells), $\R=(1,1,1,1,1,1)$, as in the hypothesis, and $\C=(1,-1,-1,1,-1,1)$.
$$
\begin{array}{r|r|r|r|r|r|r|}
&  \downarrow & \uparrow & \uparrow & \downarrow & \uparrow & \downarrow \\
\hline
 \rightarrow & \bullet & & & \bullet & \bullet & \bullet  \\ \hline
 \rightarrow  & \bullet & \bullet & & & \bullet & \bullet \\ \hline
  \rightarrow & \bullet &  \bullet & \bullet & & & \bullet \\ \hline
 \rightarrow &  \bullet  & \bullet & \bullet & \bullet & & \\ \hline
   \rightarrow & &\bullet & \bullet & \bullet & \bullet & \\ \hline
  \rightarrow & & & \bullet & \bullet & \bullet & \bullet \\ \hline
\end{array}
\quad \quad \begin{array}{r|r|r|r|r|r|r|}
&  \downarrow & \downarrow & \downarrow & \downarrow & \downarrow & \downarrow \\
\hline
 \rightarrow &  \bullet  & \bullet & \bullet & \bullet & & \\ \hline
   \leftarrow & &\bullet & \bullet & \bullet & \bullet & \\ \hline
  \leftarrow & & & \bullet & \bullet & \bullet & \bullet \\ \hline
   \rightarrow & \bullet & & & \bullet & \bullet & \bullet  \\ \hline
 \leftarrow  & \bullet & \bullet & & & \bullet & \bullet \\ \hline
  \rightarrow & \bullet &  \bullet & \bullet & & & \bullet \\ \hline
\end{array}\quad \quad \begin{array}{r|r|r|r|r|r|r|}
&  \downarrow & \downarrow & \downarrow & \downarrow & \downarrow & \downarrow \\
\hline
 \rightarrow & \bullet & & & \bullet & \bullet & \bullet  \\ \hline
 \leftarrow   & \bullet & \bullet & & & \bullet & \bullet \\ \hline
  \leftarrow  & \bullet &  \bullet & \bullet & & & \bullet \\ \hline
 \rightarrow &  \bullet  & \bullet & \bullet & \bullet & & \\ \hline
  \leftarrow  & &\bullet & \bullet & \bullet & \bullet & \\ \hline
  \rightarrow & & & \bullet & \bullet & \bullet & \bullet \\ \hline
\end{array}$$
\begin{center}
Arrays $A$, $A^t$ and $B$.
\end{center}

Note that, instead of $A^t$, we can consider the array $B$ on the right obtained from $A^t$ by a translation
on the rows of length $k-1$. We point out that $B$ has the same skeleton of $A$.
We remark that applying $(\C,\R)$ to $A^t$ is equivalent to apply $(\C, \overline{\R})$ to $B$,
where if $\R=(r_1,r_2,\ldots,r_n)$, then
$\overline{\R}=(r_k,\ldots,r_n,r_1,\ldots,r_{k-1})$.
 Since $\R=(1,1,\ldots,1)$ then also $\overline{\R}=(1,1,\ldots,1)$.
 Hence $(\C,\R)$ is a solution of $P(B)$, but
 $skel(B)=skel(A)$, so $(\C,\R)$ is a solution of $P(A)$ too.
\end{proof}

\begin{prop}\label{differentEmbeddings1}
Let $A$ and $B$ be two distinct globally simple $\H_t(m,n;h,k)$s such that $\E(A)=\E(B)$. Assume that both $A$ and $B$ admit compatible orderings and denote them, respectively, by $(\omega^A_r,\omega^A_c)$ and by $(\omega^B_r,\omega^B_c)$. Then $(\omega^A_r,\omega^A_c)$ and $(\omega^B_r,\omega^B_c)$ determine the same $k$-gonal biembedding of $K_{\frac{2nk+t}{t}\times t}$ if and only if $\omega^A_r=\omega^B_r$ and $\omega^A_c=\omega^B_c$.
\end{prop}
\begin{proof}
Suppose, by contradiction, that there exists $a\in \E(A)=\E(B)$ such that $\omega^A_r(a)\not=\omega^B_r(a)$ or $\omega^A_c(a)\not=\omega^B_c(a)$. In the following we assume, without loss of generality, that the previous condition holds for the rows.
Hence, recalling equations (\ref{ArchEmb}) and (\ref{ArchRho}),  from $\omega^A_r(a)\not=\omega^B_r(a)$, it follows that the maps $\rho^A$ and $\rho^B$ are different. Therefore $(\omega^A_r,\omega^A_c)$ and $(\omega^B_r,\omega^B_c)$ determine different $k$-gonal biembeddings of $K_{\frac{2nk+t}{t} \times t}$.

Conversely, if we have that $\omega^A_r=\omega^B_r$ and $\omega^A_c=\omega^B_c$ the maps $\rho_0^A$ and $\rho_0^B$ coincide and hence also $\rho^A=\rho^B$. In this case the compatible orderings of $A$ and of $B$ determine the same $k$-gonal biembedding of $K_{\frac{2nk+t}{t}\times t}$.
\end{proof}

\begin{rem}\label{NOlambda}
Let $A$ be a $\H_t(m,n;h,k)$. It is not hard to see that distinct solutions of $P(A)$ induce distinct
orderings $\omega_r$ and  $\omega_c$ of the rows and columns of $A$, respectively. Also,
distinct permutations determine distinct biembeddings of $K_{\frac{2nk+t}{t} \times t}$.
These facts, in general, do not hold for $^\lambda \H_t(m,n;h,k)$ with $\lambda>1$.
In the following example we show how two distinct solutions of $P(A)$, where $A$ is a $\lambda$-fold Heffter array with $\lambda >1$, induce the same permutations
$\omega_r$ and $\omega_c$.
Moreover, when $\lambda>1$, the definition of the Archdeacon embedding is more complicated
since the complete multipartite multigraph $^\lambda K_{\left(\frac{2nk}{\lambda t}+1\right)\times t}$ has repeated edges, see \cite{CPEJC}.
In this case, one could show that distinct solutions of $P(A)$ can induce the same biembedding.
\end{rem}

\begin{ex}
Let $A$ be the $^2 \H(2,5;5,2)$ of Example \ref{ex>1}.
Set $\R=(1,1)$, $\C_1=(1,1,1,1,1)$, $\C_2=(1,-1,-1,1,-1)$.
It is easy to see that $(\R,\C_1)$ and $(\R,\C_2)$ are two distinct solutions of $P(A)$.
Anyway they induce the same permutations:
  $$\begin{array}{rcl}
  \omega_r &=& (1,-2,3,4,5)(-1,2,-3,-4,-5),\\
  \omega_c &=& (1,-1)(-2,2)(3,-3)(4,-4)(5,-5).
  \end{array}$$
\end{ex}

\begin{cor}\label{differentEmbeddings2}
Let $A$ and $B$ be two  $k$-diagonal, globally simple $\H_t(n;k)$s such that:
\begin{itemize}
\item[(1)] there exists a non-empty diagonal $D_{\bar{i}}$ where $A$ and $B$ coincide;
\item[(2)] $\E(A)=\E(B)$ and $skel(A)=skel(B)$;
\item[(3)] both $P(A)$ and $P(B)$ admit a solution denoted, respectively, by $(\R_A,\C_A)$ and by $(\R_B,\C_B)$.
\end{itemize} Then $(\R_A,\C_A)$ and $(\R_B,\C_B)$ determine the same $k$-gonal biembeddings of $K_{\frac{2nk+t}{t}\times t}$ if and only if $A=B$ and $(\R_A,\C_A)=(\R_B,\C_B)$.
\end{cor}
\proof
Clearly, if $A=B$ and $(\R_A,\C_A)=(\R_B,\C_B)$, we obtain the same biembedding.

Now, assume that $(\R_A,\C_A)$ and $(\R_B,\C_B)$  determine the same $k$-gonal biembedding of $K_{\frac{2nk+t}{t}\times t}$.
We have to prove that $A=B$ and $(\R_A,\C_A)=(\R_B,\C_B)$. At this purpose we will first suppose, by contradiction, that $(\R_A,\C_A)\not=(\R_B,\C_B)$, then we will also consider the possibility that $A\not=B$. Our assumption means that either $\R_A\not=\R_B$ or $\C_A\not=\C_B$.

In the first case, there exists an index $\ell$ such that $(r_A)_{\ell}=-(r_B)_{\ell}$. Moreover, up to translate on the torus the cells of the Heffter arrays $A$ and $B$, we can assume, without loss of generality that $\ell=\bar{i}=1$.

Here we set by $\omega_r^A,\omega_c^A$ the orderings induced by $(\R_A,\C_A)$ on the elements of $\E(A)$ and by $\omega_r^B,\omega_c^B$ the orderings induced by $(\R_B,\C_B)$ on the elements of $\E(B)$.
We also denote the non-empty elements of the first row of $A$, following the natural ordering, by $(a_{1,1},a_{1,i_2},\dots, a_{1,i_k})$.
Then, since $(r_A)_1=-(r_B)_1$ and since, due to Proposition \ref{differentEmbeddings1}, $\omega_r^A=\omega_r^B$, we have that the non-empty elements of the first row of $B$ are, following the natural ordering, $(a_{1,1},a_{1,i_k},\dots, a_{1,i_2})$ where $a_{1,i_k}$ is in the $i_2$-th column and $a_{1,i_2}$ is in the $i_k$-th column.
Now we consider the element $a_{i_2,i_2}$ in position $(i_2,i_2)$ of $A$. Since, in the diagonal $D_1$, the arrays $A$ and $B$ coincide, we have that $a_{i_2,i_2}$ is also the element in position $(i_2,i_2)$ of $B$.
Here we note that, in the array $A$ the elements $a_{i_2,i_2}$ and $a_{1,i_2}$ belong both to the $i_2$-th column. On the other hand, in the array $B$ they belong to different columns: $a_{i_2,i_2}$ is in the $i_2$-th and $a_{1,i_2}$ is in the $i_k$-th. But this implies that the orbits of $a_{i_2,i_2}$ under the action of $\omega_c^A$ and $\omega_c^B$ are different and hence, due to Proposition \ref{differentEmbeddings1}, we would obtain the contradiction that $(\R_A,\C_A)$ and $(\R_B,\C_B)$  determine different biembeddings.

We obtain a similar contradiction also in the case $\C_A\not=\C_B$ and hence we have proved that $(\R_A,\C_A)=(\R_B,\C_B)$.

It is left to prove that $A=B$.
At this purpose we suppose, by contradiction, that there is a position $(\ell_1,\ell_2)$ where $A$ and $B$ are different and we consider the element $a$ of $D_{\bar{i}}$ that belongs to the $\ell_1$-th row. Due to Proposition \ref{differentEmbeddings1} we have that $\omega^A_r(a)=\omega^B_r(a)$ and, inductively, that $(\omega^A_r)^j(a)=(\omega^B_r)^j(a)$ for any $j\in [1,k]$.
Since $skel(A)=skel(B)$ and $\R_A=\R_B$, it follows that the $\ell_1$-th row of $A$ and that of $B$ are equal. But this would imply that also the elements in position $(\ell_1,\ell_2)$ of $A$ and $B$ coincide that contradicts our hypothesis. It follows that $A=B$.
\endproof

\section{On the maximum number of isomorphic embeddings}

Given an embedding $\Pi$, we will denote by $Aut(\Pi)$ the group of all automorphisms of $\Pi$ and by $Aut^+(\Pi)$ the group of the orientation preserving automorphisms. Similarly, we will denote by $Aut_0(\Pi)$ the subgroup of $Aut(\Pi)$ of the automorphisms that fix $0$ and by $Aut_0^+(\Pi)$ the group of the orientation preserving automorphisms that fix $0$. We remark that, since an orientable surface admits exactly two orientations, $Aut^+(\Pi)$ (resp. $Aut_0^+(\Pi)$) is a normal subgroup of $Aut(\Pi)$ (resp. $Aut_0(\Pi)$) whose index is either $1$ or $2$.
In the following, when we consider a $\mathbb{Z}_v$-regular embedding $\Pi$ of $\Gamma$, we identify the vertex set of $\Gamma$ with $\mathbb{Z}_v$ and we assume that the translation action is regular. We denote by $\tau_g$ the translation by $g$, i.e. the map $V(\Gamma)=\mathbb{Z}_v\rightarrow V(\Gamma)=\mathbb{Z}_v$ such that $\tau_g(x)=x+g$. Applying this convention, we have that $\tau_g\in Aut(\Pi)$ for any $g\in \mathbb{Z}_v$. Moreover, in the case of the Archdeacon embedding, recalling equation (\ref{ArchRho}), the translations also belong to $Aut^+(\Pi)$.

\begin{rem}
Let $\Pi$ and $\Pi'$ be two isomorphic $\mathbb{Z}_v$-regular embeddings of $K_{m\times t}$, where $v=mt$. Given an embedding isomorphism $\sigma: \Pi\rightarrow \Pi'$ and $g\in \mathbb{Z}_v$, we define
$$\phi_{\sigma,g}:=\sigma\circ \tau_g^{-1}\circ \sigma^{-1}\circ \tau_{\sigma(g)}.$$
Moreover, if $\sigma(0)=0$ then, since $\phi_{\sigma,g}(0)=0$, we obtain that:
$$\phi_{\sigma,g} \in Aut_0(\Pi').$$
\end{rem}

\begin{prop}\label{EqualityEmb}
Let $\Pi_0, \Pi_1$ and $\Pi_2$ be $\mathbb{Z}_v$-regular embeddings of $K_{m\times t}$, where $v=mt$.
Let us suppose there exist two embedding isomorphisms $\sigma_1:\ \Pi_1\rightarrow \Pi_0$ and $\sigma_2:\ \Pi_2\rightarrow \Pi_0$ such that, considering $\sigma_1$ and $\sigma_2$ as maps from $\mathbb{Z}_v$ to $\mathbb{Z}_v$, the following properties hold:
\begin{itemize}
\item[(1)] $\sigma_1(0)=\sigma_2(0)=0$;
\item[(2)] $\sigma_1(1)=\sigma_2(1)$;
\item[(3)] $\phi_{\sigma_1,1}=\phi_{\sigma_2,1}$.
\end{itemize}
Then the identity map from $\Pi_1$ to $\Pi_2$ is an isomorphism.
\end{prop}
\begin{proof}
We note that, due to hypothesis $(3)$, we have that:
$$\phi_{\sigma_1,1}=\sigma_1\circ \tau_1^{-1}\circ \sigma_1^{-1}\circ \tau_{\sigma_1(1)}=\sigma_2\circ \tau_1^{-1}\circ \sigma_2^{-1}\circ \tau_{\sigma_2(1)}=\phi_{\sigma_2,1}.$$
Since, because of hypothesis $(2)$, $\sigma_1(1)=\sigma_2(1)$ the maps $\tau_{\sigma_1(1)}$ and $\tau_{\sigma_2(1)}$ coincide. Reducing these maps from the composition, we obtain that:
\begin{equation}\label{eq7}
\sigma_1\circ \tau_1^{-1}\circ \sigma_1^{-1}=\sigma_2\circ \tau_1^{-1}\circ \sigma_2^{-1}.
\end{equation}
Note that equation (\refeq{eq7}) can be rewritten as:
$$(\sigma_2^{-1}\circ \sigma_1)\circ \tau_1^{-1}= \tau_1^{-1}\circ (\sigma_2^{-1}\circ \sigma_1),$$
hence we have that:
\begin{equation}\label{eq1}
\tau_1\circ (\sigma_2^{-1}\circ \sigma_1)=(\sigma_2^{-1}\circ \sigma_1)\circ \tau_1.
\end{equation}
Setting $\sigma_{1,2}:=\sigma_2^{-1}\circ \sigma_1$, by definition of $\tau_1$, it results
$$\tau_1\circ \sigma_{1,2}(x)=\sigma_{1,2}(x)+1,$$
and
$$\sigma_{1,2}\circ \tau_1(x)=\sigma_{1,2}(x+1). $$
Therefore, equation (\refeq{eq1}) can be written as:
$$\sigma_{1,2}(x+1)=\sigma_{1,2}(x)+1.$$
Since, for hypothesis (1), $\sigma_1(0)=\sigma_2(0)=0$ we can prove, inductively, that $\sigma_{1,2}(x)=x$ that is $\sigma_{1,2}=id$.
It follows that the identity map from $\Pi_1$ to $\Pi_2$ is an isomorphism of embeddings.
\end{proof}

\begin{prop}\label{CarAut}
Let $\Pi$ be an embedding of $K_{m\times t}$ where $m\geq 2$. Then we have that:
$$|Aut_0(\Pi)|\leq 2|Aut_0^+(\Pi)|\leq 2|N(K_{m\times t},0)|=2(m-1)t.$$
\end{prop}

\begin{proof}
Since $Aut^+_0(\Pi)$ is a normal subgroup of $Aut_0(\Pi)$ whose index is at most two, it suffices to prove that $|Aut_0^+(\Pi)|\leq |N(K_{m\times t},0)|.$ Because of the definition, $\sigma \in Aut_0^+(\Pi)$ implies that, for any $x\not=0$:
$$\sigma\circ \rho(0,x)=\rho\circ \sigma(0,x).$$
Recalling that $\rho(0,x)=(0,\rho_0(x))$ for a suitable map $\rho_0: N(K_{m\times t},0)\rightarrow N(K_{m\times t},0)$, we have that:
\begin{equation}\label{eq2}\sigma\circ \rho(0,x)=(0,\sigma\circ \rho_0(x))=(0,\rho_0\circ \sigma(x))=\rho\circ \sigma(0,x).\end{equation}
Since $|N(K_{m\times t},0)|=(m-1)t$, we can write $\rho_0$ as the cycle $(x_1=1,x_2,x_3,\dots,x_{(m-1)t})$.
Then, setting $\sigma(x_1)=x_i$, equation (\refeq{eq2}) implies that:
$$(0,\sigma(x_2))=(0,\rho_0\circ \sigma(x_1))=\rho\circ \sigma(0,x_1)=(0,x_{i+1}).$$
Therefore, we can prove, inductively, that:
$$\sigma(x_j)=x_{j+i-1}$$
where the indices are considered modulo $(m-1)t$.
This means that $\sigma|_{N(K_{m\times t},0)}=\rho_0^{i-1}$ and that $\sigma$ is fixed in $N(K_{m\times t},0)$ when the image of one element is given.
In particular since $\rho_0$ has order $(m-1)t$, there are at most $|N(K_{m\times t},0)|$ possibilities for the map $\sigma|_{N(K_{m\times t},0)}$.

Now we need to prove that, if two automorphisms $\sigma_1$ and $\sigma_2$ of $Aut_0(\Pi)$ coincide in $N(K_{m\times t},0)$, they coincide everywhere. Set $\sigma_{1,2}=\sigma_2^{-1}\circ \sigma_1$, this is equivalently to prove that $\sigma_{1,2}$ is the identity. Given $x\in N(K_{m\times t},0)$ we have that $\sigma_{1,2}(x)=x$ and hence $\sigma_{1,2}$ belongs to the subgroup $Aut_x^+(\Pi)$ of $Aut^+(\Pi)$ of the elements that fix $x$. Proceeding as before we prove that $\sigma_{1,2}|_{N(K_{m\times t},x)}$ is fixed when the image of one element is given.
But now we note that $0\in N(K_{m\times t},x)$ and we have that $\sigma_{1,2}(0)=0$. It follows that
$$\sigma_{1,2}|_{N(K_{m\times t},x)}=id.$$
Since $\sigma_1$ and $\sigma_2$ coincide in $N(K_{m\times t},0)$, we also have that
$$\sigma_{1,2}|_{N(K_{m\times t},0)}=id.$$
Now the thesis follows because, for $m\geq 2$,
$$V(K_{m\times t})=N(K_{m\times t},0)\cup N(K_{m\times t},x).$$
\end{proof}

\begin{prop}\label{upperboundFamily}
Let $\mathcal{F}=\{\Pi_{\alpha}: \alpha \in \mathcal{A}\}$ be a family of $\mathbb{Z}_v$-regular distinct embeddings of $K_{m\times t}$ where $v=mt$ and  $m\geq 2$. Then, if $\Pi_{\alpha}$ is isomorphic to $\Pi_0$ for any $\alpha \in \mathcal{A}$, we have that:
$$|\mathcal{F}|\leq 2|Aut_0(\Pi_0)|\cdot |N(K_{m\times t},0)|\leq 4|N(K_{m\times t},0)|^2=4((m-1)t)^2.$$
Moreover, if for any $\alpha\in \mathcal{A}$ and any $g\in \mathbb{Z}_v$, the translation $\tau_g$ belongs to $Aut^+(\Pi_{\alpha})$,
then:
$$|\mathcal{F}|\leq 2|Aut_0^+(\Pi_0)|\cdot |N(K_{m\times t},0)|\leq 2|N(K_{m\times t},0)|^2=2((m-1)t)^2.$$
\end{prop}
\begin{proof}
We can assume $\Pi_0\in \mathcal{F}$ and let us denote by $\sigma_\alpha$ an isomorphism between $\Pi_\alpha$ and $\Pi_0$ that fixes $0$. Note that this isomorphism exists since $\mathcal{F}$ is a family of $\mathbb{Z}_v$-regular embeddings.
Let us assume, by contradiction that
$$|\mathcal{F}|> 2|Aut_0(\Pi_0)|\cdot|N(K_{m\times t},0)|.$$
We note that, for any $\alpha \in \mathcal{A}$, $\phi_{\sigma_{\alpha},1}\in Aut_0(\Pi_0)$. Since $\sigma_{\alpha}$ is an isomorphism that fixes $0$, $\sigma_{\alpha}(1)$ belongs to $N(K_{m\times t},0)$ if and only if $1$ belongs to $N(K_{m\times t},0)$. It follows that, we have at most
$$\max(|N(K_{m\times t},0)|,v-1-|N(K_{m\times t},0)|)=\max((m-1)t,t-1)=(m-1)t$$
possibilities for $\sigma_{\alpha}(1)$.
Therefore, due to the pigeonhole principle, we would have that there exist $\Pi_1$, $\Pi_2$ and $\Pi_3$ in $\mathcal{F}$ such that:
\begin{itemize}
\item[(1)] $\sigma_1(1)=\sigma_2(1)=\sigma_3(1)$;
\item[(2)] $\phi_{\sigma_1,1}=\phi_{\sigma_2,1}=\phi_{\sigma_3,1}$.
\end{itemize}
Hence, due to Proposition \ref{EqualityEmb}, we would have that the identity is an isomorphism both from $\Pi_1=(\Gamma_1,\rho_1)$ to $\Pi_2=(\Gamma_2,\rho_2)$ and from $\Pi_1=(\Gamma_1,\rho_1)$ to $\Pi_3=(\Gamma_3,\rho_3)$. It follows from Definition \ref{DefEmbeddings} that $\Gamma_1=\Gamma_2=\Gamma_3$ and $\rho_2, \rho_3\in \{\rho_1, \rho_1^{-1}\}$. But this means that either $\Pi_1=\Pi_2$ or $\Pi_1=\Pi_3$ or $\Pi_2=\Pi_3$. In each of these cases we would obtain that the elements of $\mathcal{F}$ are not all distinct that contradicts the hypotheses.

We remark that, in case the translations are all elements of $Aut^+(\Pi_{\alpha})$ (for every $\alpha\in \mathcal{A}$), $\phi_{\sigma_{\alpha},1}$ would be an element of $ Aut_0^+(\Pi_0)$ and hence we can substitute $Aut_0(\Pi_0)$ with $Aut_0^+(\Pi_0)$ in the previous argument. This leads us to obtain:
$$|\mathcal{F}|\leq 2|Aut_0^+(\Pi_0)|\cdot |N(K_{m\times t},0)|.$$
\end{proof}

\begin{rem}
Clearly if $t=1$ the complete multipartite graph $K_{m \times t}$ is nothing but the complete graph of order $m$.
Hence the results of Propositions \ref{CarAut} and \ref{upperboundFamily} hold also for the complete graph.
\end{rem}

\section{Embeddings from Cavenagh, Donovan and Yaz\i c\i's arrays}
We consider now the family of embeddings of $K_v$ obtained by Cavenagh, Donovan, and Yaz\i c\i \ in \cite{CDDYbiem}. In their constructions, all the face boundaries are cycles of length $k$.

Set the binary entropy function by $H(p):=-p\log_2{p}-(1-p)\log_2(1-p)$ and denoted by $\mathcal{H}(m)$ the cardinality of the derangements on $[0,m-1]$, we will use the following, well known, approximations:
\begin{equation}\label{Stirling}
m!\approx \sqrt{2m\pi}\left(\frac{m}{e}\right)^m,
\end{equation}
\begin{equation}\label{Entropy}
{m\choose pm }\approx \frac{1}{\sqrt{2m\pi(1-p)p}}2^{mH(p)},
\end{equation}
\begin{equation}\label{Derangements}
\mathcal{H}(m)\approx m!/e,
\end{equation}
where the symbol $\approx$ means that the two quantities are asymptotic: their ratio tends to $1$ as $m$ tends to infinity. We will also use the simbol $\gtrsim$ in case the $\liminf$ of the ratio between two quantities, as $m$ tends to infinity, is greater than or equal to $1$.
\begin{thm}[Cavenagh, Donovan and Yaz\i c\i \cite{CDDYbiem}]\label{CDY}
Let $v=2nk+1$, $k=4t+3$ and let $n\equiv 1 \pmod{4}$ be either a prime or $n\geq (7k+1)/3$. Moreover, if $n\equiv 0\pmod{3}$, we also assume that $k\equiv 7\pmod{12}$. Then, the number of distinct simple $k$-gonal biembeddings of $K_v$ is, at least, of:
$$(n-2)[\mathcal{H}(t-2)]^2\approx (n-2)[(t-2)!/e]^2.$$
Also, for all such embeddings and all $g\in \mathbb{Z}_v$, $\tau_g$ is an orientation preserving automorphism.
\end{thm}
Using Proposition \ref{upperboundFamily} and Theorem \ref{CDY}, we can prove the following result.

\begin{thm}\label{GeneralBound}
Let $v=2nk+1$, $k=4t+3$ and let $n\equiv 1 \pmod{4}$ be either a prime or $n\geq (7k+1)/3$. Moreover, if $n\equiv 0\pmod{3}$, we also assume that $k\equiv 7\pmod{12}$. Then, the number of non-isomorphic simple $k$-gonal biembeddings of $K_v$ is, at least, of:
$$\frac{(n-2)[\mathcal{H}(t-2)]^2}{2(2nk)^2}\approx \frac{\pi(t-2)^{2t-5}}{64e^{2t-2}n}\approx k^{k/2+o(k)}/v.$$
\end{thm}
\begin{proof}
Let us consider, for given $k$ and $v$, the distinct simple $k$-gonal biembeddings of $K_v$ provided in \cite{CDDYbiem}. Let us partition these embeddings into families of isomorphic ones. The thesis easily follows because, due to Proposition \ref{upperboundFamily}, each of these families has size at most $2(v-1)^2=8(nk)^2$. Then the lower bound on the number of non-isomorphic simple $k$-gonal biembeddings of $K_v$ can be approximated using the Stirling formula for the factorial, that is equation (\refeq{Stirling}), and the approximation (\refeq{Derangements}).
\end{proof}
Now we will show that, studying carefully the Crazy Knight's Tour Problem for the Heffter arrays found by Cavenagh, Donovan and Yaz\i c\i\ it is possible to get many other simple $k$-gonal biembeddings of $K_v$.

We consider here a $k$-diagonal array $A$ of size $n> k$ and vectors $\R=(1,\dots,1)$ and $\C\in \{-1,1\}^n$, whose $-1$ are in positions $E=(e_1,\dots,e_r)$ where $e_1<e_2<\dots<e_r$.
We state a characterization, obtained with the same proof of Lemma 4.19 of \cite{CDP}, of the solutions of $P(A)$ that have a trivial vector $\R$, i.e. $\R=(1,\dots,1)$.
\begin{lem}\label{Percorso1bis}
Let $k\geq3$ be an odd integer and let $A$ be a $k$-diagonal array of size $n> k$, widths $s_1,s_2,\dots, s_i$ and with non-empty diagonal $D_1$.
Then the vectors $\R=(1,\dots,1)$ and $\C\in \{-1,1\}^n$, where the positions of each $-1$ in $\C$ are described by $E$, are a solution of $P(A)$ if and only if:
\begin{itemize}
\item[(1)] for any $j\in [1,n]$, the list $E$ covers all the congruence classes modulo $d_j$, where $d_j=\gcd(n,s_j)$;
\item[(2)] the list $ L_{\R,\C}(1,1)$ covers all the positions of $\{(e,e)| e\in E\}$.
\end{itemize}
\end{lem}

\begin{prop}\label{solutionPrime}
Let $k$ be an odd integer, $n>8k$ be a prime, and let $A$ be a $k$-diagonal Heffter array $\H(n;k)$ whose filled diagonals are $D_1,D_2,\dots,D_{k-3},D_{k-1},D_{k},D_{k+1}$. Then, the number of distinct solutions of $P(A)$ is at least of:
$$ 2\binom{\lceil n/2k\rceil}{ \lceil n/8k\rceil}\gtrsim \frac{\sqrt{k}}{\sqrt{3\pi n}} 2^{\frac{n}{2k}\cdot H(1/4)+3}.$$
\end{prop}
\begin{proof}
Let us consider a subset $E=(e_1,\dots,e_r)$
where $e_1<e_2<\dots<e_r$ of $[1,n]$ that satisfies the following properties:
\begin{itemize}
\item[(1)] the elements $e_1,\dots,e_r$ of $E$ are integers equivalent to $1$ modulo $2k$;
\item[(2)] $r=|E|$ is coprime with $k-2$.
\end{itemize}
A set $E$ with such properties can be constructed as follows. Let $r$ be a prime in the range $[\frac{n}{8k}, \frac{n}{4k}]$ that exists because of Bertrand's postulate. Then we choose $r$ elements $e_1,\dots,e_r$ among the $\lceil n/2k\rceil$ integers equivalent to $1$ modulo $2k$ contained in $[1,n]$. The number of such choices is at least of
$$\binom{\lceil n/2k\rceil}{r}\geq \binom{\lceil n/2k\rceil}{ \lceil n/8k\rceil}.$$
Note that, due to the approximation for the binomial coefficients, see equation (\refeq{Entropy}), this number can be approximated to
$$ \binom{\lceil n/2k\rceil}{ \lceil n/8k\rceil}\gtrsim \frac{\sqrt{k}}{\sqrt{3\pi n}} 2^{\frac{n}{2k}\cdot H(1/4)+2}.$$

Hence, in order to obtain the thesis, it suffices to prove that, set $\R=(1,1,\dots,1)$ and $\C_E\in \{ -1,1\}^n$ whose $-1$ are in positions $E=(e_1,\dots,e_r)$, $(\R,\C_E)$ is a solution for $P(A)$. Indeed, according to Lemma \ref{-R-C}, the number of distinct solutions of $P(A)$ would be, at least, of
$$ 2\binom{\lceil n/2k\rceil}{ \lceil n/8k\rceil}\gtrsim \frac{\sqrt{k}}{\sqrt{3\pi n}} 2^{\frac{n}{2k}\cdot H(1/4)+3}.$$
Since $n$ is a prime, condition $(1)$ of Lemma \ref{Percorso1bis} is satisfied. We need to check that also condition $(2)$ of the same lemma holds. At this purpose, we consider an element $(e,e)\in D_1$ with $e\in E$, then there exists a minimum $m\geq 1$ such that $S_{\R,\C}^m((e,e))=(e',e')$ for some $e'\in E$.
We define the permutation $\omega_{\C}$ on $E$ as $\omega_{\C}(e)=e'$.
We need to prove that $\omega_{\C}$ is a cycle of order $r$.
Given $e\in E$, the second cell of the form $(e',e')$ with $e'\in E$ we meet in the list $ L_{\R,\C}(e,e)$ is reached after the following moves:
\begin{itemize}
\item[{[1]}] from $(e,e)$ we move backward into the diagonal $D_1$ with steps of length $k$ until we reach a cell of the form $(e_i+k,e_i+k)$ with $e_i\in E$;
\item[{[2]}] from $S_{\R,\C}(e_i+k,e_i+k)=(e_i+(k-1),e_i)$ we move forward into the diagonal $D_{k}$ with steps of length $1$ until we reach the cell $(e_{i+1}-1+(k-1),e_{i+1}-1)$, where the indices are considered modulo $r$ (as for the rest of this proof);
\item[{[3]}] from $S_{\R,\C}(e_{i+1}-1+(k-1),e_{i+1}-1)=(e_{i+1}+(k-4),e_{i+1})$ we move forward into the diagonal $D_{k-3}$ with steps of length $1$ until we reach the cell $(e_{i+2}-1+(k-4),e_{i+2}-1)$; we reiterate this procedure into the diagonals $D_{k-5},D_{k-7},\dots,D_4$;
\item[{[4]}] since $k$ is odd, we arrive to the cell $(e_{i+(k-3)/2}+1,e_{i+(k-3)/2})\in D_2$ from which we move forward with steps of length $1$ until we reach the cell $(e_{i+(k-1)/2},e_{i+(k-1)/2}-1)$;
\item[{[5]}] from $S_{\R,\C}(e_{i+(k-1)/2},e_{i+(k-1)/2}-1)=(e_{i+(k-1)/2}+k,e_{i+(k-1)/2})$ we move forward into the diagonal $D_{k+1}$ with steps of length $1$ until we reach the cell $(e_{i+(k+1)/2}-1+k,e_{i+(k+1)/2}-1)$; we reiterate this procedure into the diagonals $D_{k-1} $ (here with steps of length $2$),$D_{k-4},\dots,D_3$;
\item[{[6]}] since $k$ is odd, we arrive to the cell $(e_{i+(k-1)},e_{i+(k-1)})\in D_1$ that is the second one of the form $(e',e')\in D_1$ with $e'\in E$ we meet in the list $L_{\R,\C}(e,e)$.
\end{itemize}
We denote by $\gamma$ the cyclic permutation of the elements of $E$ defined by $(e_1,\dots,e_r)$. We note that since the distances between elements of $E$ are multiples of $k$, in the first step of the above procedure we apply the permutation $\gamma^{-1}$. Then, from the previous discussion, it follows that $\omega_{\C}=\gamma^{k-1}\circ \gamma^{-1}=\gamma^{k-2}.$
Since $r$ is coprime with $k-2$ and $\gamma$ is a cycle of order $r$, then $\omega_{\C}$ is also a cycle of order $r$ and hence condition $(2)$ of Lemma \ref{Percorso1bis} is satisfied.
\end{proof}

\begin{rem}\label{MigliorareEsponente}
We note that, if $n$ is sufficiently large, in the proof of Proposition \ref{solutionPrime}, the choice of $r$ could also be done in the range $[\lambda\frac{n}{k}, \frac{n}{4k}]$ where $\lambda$ is smaller than $1/4$. In fact, if $|\frac{n}{4k}-\lambda\frac{n}{k}|\geq k-2$, we can find $r$ coprime with $k-2$ also in this range. It follows that, given $\lambda<1/4$, we can replace the exponent $\frac{n}{2k}\cdot H(1/4)$ of the previous proposition with $\frac{n}{2k}\cdot H(2\lambda)$. However, due to the complications in the notations, we believe it is better to write the statement in the ``clearest'' case.
\end{rem}

\begin{thm}\label{CDY2}
Let $v=2nk+1$, $k=4t+3$ and let $n\equiv 1\pmod{4}$ be a prime greater than $8k$. Then the number of distinct simple $k$-gonal biembeddings of $K_v$ is, at least, of:
$$2(n-2)[\mathcal{H}(t-2)]^2 \binom{\lceil n/2k\rceil}{ \lceil n/8k\rceil} \gtrsim \frac{[(t-2)!]^2\sqrt{(4t+3)n}}{e^2\sqrt{3\pi}} 2^{\frac{n}{2(4t+3)}\cdot H(1/4)+3}.$$
Also, for all such embeddings and all $g\in \mathbb{Z}_v$, $\tau_g$ is an orientation preserving automorphism.
\end{thm}
\begin{proof}
We note that, if $n$ is a prime, each array of the family $\mathcal{F}_{n,k}:=\{A_i:\ i\in \mathcal{A}_{n,k}\}$ of globally simple $H(n;k)$s constructed in \cite{CDDYbiem} satisfies (setting $\alpha=2p+2$) the hypotheses of Proposition \ref{solutionPrime}. Therefore, for each array $A_i$ of $\mathcal{F}_{n,k}$ the number of solutions of $P(A_i)$ is at least of: $$ 2\binom{\lceil n/2k\rceil}{ \lceil n/8k\rceil}\gtrsim \frac{\sqrt{k}}{\sqrt{3\pi n}} 2^{\frac{n}{2k}\cdot H(1/4)+3}.$$

We also recall that, due to Theorem \ref{CDY}, the number of such arrays is at least of:
$$(n-2)[\mathcal{H}(t-2)]^2\approx (n-2)[(t-2)!/e]^2.$$

Now we note that, given $n$ and $k$, these arrays have all the same entries and skeleton and coincide in at least $5$ diagonals. Therefore, because of Corollary \ref{differentEmbeddings2}, however we take $A_i\in \mathcal{F}_i$ and a solution of $P(A_i)$, we determine a different embedding.

It follows that the number of distinct simple $k$-gonal biembeddings of $K_v$ is, at least, of:
$$2(n-2)[\mathcal{H}(t-2)]^2 \binom{\lceil n/2k\rceil}{ \lceil n/8k\rceil} \gtrsim \frac{[(t-2)!]^2\sqrt{kn}}{e^2\sqrt{3\pi}} 2^{\frac{n}{2k}\cdot H(1/4)+3}.$$
\end{proof}

By Proposition \ref{upperboundFamily} and Theorem \ref{CDY2}, it follows that:
\begin{thm}\label{CDY3}
Let $v=2nk+1$, $k=4t+3$ and let $n\equiv 1 \pmod{4}$ be a prime greater than $8k$. Then the number of non-isomorphic simple $k$-gonal biembeddings of $K_v$ is, at least, of:
$$\frac{(n-2)}{(2nk)^2}[\mathcal{H}(t-2)]^2 \binom{\lceil n/2k\rceil}{ \lceil n/8k\rceil} \approx \frac{[(t-2)!]^2}{e^2\sqrt{3\pi(n(4t+3))^3}} 2^{\frac{n}{2(4t+3)}\cdot H(1/4)}\approx k^{\frac{k}{2}+o(k)} \cdot 2^{v\cdot \frac{H(1/4)}{(2k)^2}+o(v)}.$$
\end{thm}

\begin{prop}\label{solutionGeneral}
Let $k$ be an odd integer and let $A$ be a $k$-diagonal Heffter array $\H(n;k)$ whose filled diagonals are $D_1,D_2,\dots,D_i,D_{i+s_1},D_{i+s_1+2},D_{i+s_1+3},\dots,D_{k+s_1}$. Assuming that $\gcd(n,2)=\gcd(n,s_1)=\gcd(n,k+s_1-1)=1$, the number of distinct solutions of $P(A)$ is at least of $ 2\binom{n}{2}.$
\end{prop}
\begin{proof}
Let us consider a subset $E=(e_1,e_2)$
where $e_1<e_2$ of $[1,n]$. Hence in order to obtain the thesis, it suffices to prove that, set $\R=(1,1,\dots,1)$ and $\C_{E}\in \{ -1,1\}^n$ whose $-1$ are in positions $E=(e_1,e_2)$, $(\R,\C_{E})$ is a solution for $P(A)$. Indeed, due to Lemma \ref{-R-C}, the number of distinct solutions of $P(A)$ would be, at least, of $ 2\binom{n}{2}$.
Since $n$ is coprime with $2$, $s_1$ and $k+s_1-1$, condition $(1)$ of Lemma \ref{Percorso1bis} is satisfied. We need to check that also condition $(2)$ holds. Defined $\omega_{\C}$ and $\gamma$ as in the proof of Proposition \ref{solutionPrime}, we obtain that, also here, $\omega_{\C}=\gamma^{k-2}$.
Since $\gamma$ is a cycle of order $2$ and $k-2$ is odd, $\omega_{\C}$ is also a cycle of order $2$. Hence condition $(2)$ of Lemma \ref{Percorso1bis} is satisfied.
\end{proof}
\begin{thm}\label{CDY4}
Let $v=2nk+1$, $k=4t+3$ and let $n\equiv 1 \pmod{4}$ be such that $n\geq (7k+1)/3$. Moreover, if $n\equiv 0\pmod{3}$, we also assume that $k\equiv 7\pmod{12}$. Then the number of distinct simple $k$-gonal biembeddings of $K_v$ is, at least, of:
$$2(n-2)\binom{n}{2}[\mathcal{H}(t-2)]^2\approx \frac{n^3[(t-2)!]^2}{e^2}.$$
Also, for all such embeddings and all $g\in \mathbb{Z}_v$, $\tau_g$ is an orientation preserving automorphism.
\end{thm}
\begin{proof}
The thesis follows from Proposition \ref{solutionGeneral} reasoning as in the proof of Theorem \ref{CDY2}.
\end{proof}
From Proposition \ref{upperboundFamily}, it follows that:
\begin{thm}\label{CDY5}
Let $v=2nk+1$, $k=4t+3$ and let $n\equiv 1 \pmod{4}$ be such that $n\geq (7k+1)/3$. Moreover, if $n\equiv 0\pmod{3}$, we also assume that $k\equiv 7\pmod{12}$. Then the number of non-isomorphic simple $k$-gonal biembeddings of $K_v$ is, at least, of:
$$\frac{n-2}{(2nk)^2}\binom{n}{2}[\mathcal{H}(t-2)]^2\approx \frac{n[(t-2)!]^2}{8((4t+3)e)^2}\approx v \cdot k^{k/2+o(k)}.$$
\end{thm}

\section{Embeddings from cyclically $k$-diagonal Heffter arrays}
We note that the bounds obtained in Theorems \ref{CDY3} and \ref{CDY5} grow more than exponentially in $k$ but they have some restrictions on the considered values of $n$. Furthermore, they grow exponentially in $n$ only when $n$ is a prime. For this reason, in this section, we will provide lower bounds that grow exponentially in $n$ on the number of $k$-gonal biembeddings not only of complete graphs but also of complete multipartite graphs.

First of all, we need to recall the following existence result reported in \cite{MP3} (see Corollaries 3.4 and 3.6) on cyclically $k$-diagonal Heffter arrays.
\begin{lem}\label{ExistenceDiagonal}
Given $n\geq k\geq 3$, then there exists a cyclically $k$-diagonal Heffter array $H_t(n;k)$ in each of the following cases:
\begin{itemize}
\item[(1)] $t\in \{1,2\}$ and $k\equiv 0\pmod{4}$ \cite{ADDY, MP};
\item[(2)] $t\in \{1,2\}$, $k\equiv 1\pmod{4}$ and $n\equiv 3\pmod{4}$  \cite{CMPPHeffter,DW};
\item[(3)] $t\in \{1,2\}$, $k\equiv 3\pmod{4}$ and $n\equiv 0,1\pmod{4}$ \cite{ADDY};
\item[(4)] $t=k$, $k\equiv 1\pmod{4}$ and $n\equiv 3\pmod{4}$ \cite{RelH};
\item[(5)] $t=k$, $k\equiv 3\pmod{4}$ and $n\equiv 0,3\pmod{4}$ \cite{RelH};
\item[(6)] $t\in \{n,2n\}$, $k=3$ and $n$ is odd \cite{CPPBiembeddings}.
\end{itemize}
\end{lem}
Moreover, in \cite{CMPPHeffter} and in \cite{CPPBiembeddings}, it is also proved the following existence result on globally simple cyclically $k$-diagonal Heffter arrays.
\begin{lem}\label{ExistenceDiagonalGS}
Given $n\geq k\geq 3$, then there exists a globally simple, cyclically $k$-diagonal Heffter array $H_t(n;k)$ in each of the following cases:
\begin{itemize}
\item[(1)] $t\in \{1,2\}$, $k\in \{3,5,7,9\}$ and $nk\equiv 3 \pmod{4}$;
\item[(2)] $t=k$, $k\in \{3,5,7,9\}$ and $n\equiv 3 \pmod{4}$;
\item[(3)] $t\in \{n,2n\}$, $k=3$ and $n$ is odd.
\end{itemize}
\end{lem}
The goal will be now to find an exponential family of solutions of $P(A)$ where $A$ is one of those arrays and then to proceed by using the following remark.
\begin{rem}\label{NonIsoSol}
Let us assume we have $M$ distinct solutions of $P(A)$ where $A$ is a given (globally simple) cyclically $k$-diagonal Heffter array $\H_t(n;k)$ and $k$ is an odd integer. In this case we may assume, without loss of generality, that the filled diagonals are $D_1,D_2,\dots,D_{(k+1)/2}$ and $D_{n},D_{n-1},\dots,D_{n-(k-3)/2}$. Then, if we consider $A^t$, we have that $skel(A)=skel(A^t)$ and hence any solution of $P(A)$ is also a solution of $P(A^t)$. Moreover, $A$ and $A^t$ coincide on $D_1$ and $\E(A)=\E(A^t)$. Therefore, due to Corollary \ref{differentEmbeddings2}, there are at least $2M$ distinct $\Z_{2nk+t}$-regular (simple) $k$-gonal biembeddings of $K_{\frac{2nk+t}{t}\times t}$.
Then, because of Proposition \ref{upperboundFamily}, the number of non-isomorphic (simple) $k$-gonal biembeddings of $K_{\frac{2nk+t}{t}\times t}$ is, at least, of $\frac{M}{(2nk)^2}.$
\end{rem}

For a cyclically $k$-diagonal array $A$, we recall a characterization, provided in \cite{CDP}, of the solutions of $P(A)$ that have vector $\R=(1,\dots,1)$.

We consider here a cyclically $k$-diagonal array of size $n> k$ and vectors $\R=(1,\dots,1)$ and $\C\in \{-1,1\}^n$, whose $-1$ are in positions $E=(e_1,\dots,e_r)$
where $e_1<e_2<\dots<e_r$. We note that, given $e\in E$, there exists a minimum $m\geq 1$ such that $e-m(k-1)\equiv e''\pmod{n}$ for some $e''\in E$.
We define the permutation $\omega_{1,\C}$ on $E$ as $\omega_{1,\C}(e)=e''$.
Finally we define the permutation $\omega_{2,C}$ on $E=(e_1,\ldots,e_r)$ as $\omega_{2,C}(e_i)=e_{i+(k-1)}$ where the indices are considered modulo $r$. Then, in \cite{CDP}, it is proven that:

\begin{lem}\label{Percorso2}
Let $k\geq3$ be an odd integer and let $A$ be a cyclically $k$-diagonal array of size $n> k$.
Then the vectors $\R:=(1,\dots,1)$ and $\C\in \{ -1,1\}^n$, whose $-1$ are in positions $E=(e_1,\dots,e_r)$
where $e_1<e_2<\dots<e_r$, are a solution of $P(A)$ if and only if:
\begin{itemize}
\item[(1)] the list $E$ covers all the congruence classes modulo $d$, where $d=\gcd(n,k-1)$;
\item[(2)] the permutation $\omega_{2,\C}\circ \omega_{1,\C}$ on $E$ is a cycle of order $r=|E|$.
\end{itemize}
\end{lem}

\begin{prop}\label{3diag}
Let $A$ be a cyclically $3$-diagonal Heffter array $\H_t(n;3)$ where $n\geq 3$ is an odd integer. Then the number of distinct solutions of $P(A)$ is, at least, of $2^{\frac{1}{2}n+2}$.
\end{prop}
\begin{proof}
Let us consider a subset $E=(e_1,\dots,e_r)$
of $[1,n]$ where the elements $e_1,\dots,e_r$ are odd integers such that $e_1<e_2<\dots<e_r$.

We note that the set $O$ of odd elements in $[1,n]$ has cardinality $\frac{n+1}{2}>\frac{1}{2}n.$ The number of subsets of $O$ is then at least of $2^{\frac{1}{2}n}$. It follows that the number of possible choices for $E$ is, at least, of $ 2^{\frac{1}{2}n}$. Hence, in order to obtain the thesis, it suffices to prove that, set $\R=(1,1,\dots,1)$ and $\C_E\in \{ -1,1\}^n$ whose $-1$ are in positions $E=(e_1,\dots,e_r)$, $(\R,\C_E)$ is a solution for $P(A)$. Indeed, due to Lemmas \ref{-R-C} and \ref{CR}, the number of distinct solutions of $P(A)$ would be, at least, of $2^{\frac{1}{2}n+2}$.

Here we denote by $\gamma$ the cyclic permutation of the elements of $E$ defined by $(e_1,\dots,e_r)$.
In this case, since $k-1=2$ we have that $\omega_{2,\C}=\gamma^2$. Similarly, since the elements of $E$ are all odd integers, $\omega_{1,\C}=\gamma^{-1}$. It follows that $\omega_{2,\C}\circ \omega_{1,\C}=\gamma.$
Since $n$ is odd, we also have that $d=\gcd(n,2)=1$ and hence both the conditions of Lemma \ref{Percorso2} are satisfied and $(\R,\C_E)$ is a solution of $P(A)$.
\end{proof}
We can also provide a similar construction for arbitrary odd $k$. In this case, we still obtain an exponential lower bound to the number of solutions of $P(A)$ but, here, if we consider $k=3$, the exponent is worse than that of Proposition \ref{3diag}.

\begin{prop}\label{power2}
Let $A$ be a cyclically $k$-diagonal Heffter array $\H_t(n;k)$ where $n\geq 4k-3$ and $k$ are odd integers such that $\gcd(n,k-1)=1$. Then the number of distinct solutions of $P(A)$ is, at least, of
$$ 4\binom{\lceil n/(k-1)\rceil}{ \lceil n/(4k-4)\rceil}\gtrsim \sqrt{\frac{ 2(k-1)}{3n\pi}}2^{\frac{n}{k-1}\cdot H(1/4)+3}.$$
\end{prop}
\begin{proof}
Let us consider a subset $E=(e_1,\dots,e_r)$ of $[1,n]$, where $e_1<e_2<\dots<e_r$, that satisfies the following properties:
\begin{itemize}
\item[(1)] the elements $e_1,\dots,e_r$ of $E$ are integers equivalent to $1$ modulo $k-1$;
\item[(2)] $r=|E|$ is an integer coprime with $k-2$.
\end{itemize}
A set $E$ with such properties can be constructed as follows. Let $r$ be a prime in the range $[\frac{n}{4(k-1)}, \frac{n}{2(k-1)}]$ that exists because of Bertrand's postulate. Then we choose $r$ elements $e_1,\dots,e_r$ among the $\lceil n/(k-1)\rceil$ integers equivalent to $1$ modulo $k-1$ contained in $[1,n]$. The number of such choices is at least of
$$\binom{\lceil n/(k-1)\rceil}{r}\geq \binom{\lceil n/(k-1)\rceil}{ \lceil n/(4k-4)\rceil}.$$
Note that, due to the approximation for the binomial coefficients, see equation (\refeq{Entropy}), this number can be so approximated
$$ \binom{\lceil n/(k-1)\rceil}{ \lceil n/(4k-4)\rceil}\gtrsim \sqrt{\frac{ 8(k-1)}{3n\pi}}2^{\frac{n}{k-1}\cdot H(1/4)}.$$
Hence, also here, in order to obtain the thesis, it suffices to prove that, set $\R=(1,1,\dots,1)$ and $\C_{E}\in \{ -1,1\}^n$ whose $-1$ are in positions $E=(e_1,\dots,e_r)$, $(\R,\C_E)$ is a solution for $P(A)$. Indeed, due to Lemmas \ref{-R-C} and \ref{CR}, the number of distinct solutions of $P(A)$ would be, at least, of
$$ 4\binom{\lceil n/(k-1)\rceil}{ \lceil n/(4k-4)\rceil}\gtrsim \sqrt{\frac{ 2(k-1)}{3n\pi}}2^{\frac{n}{k-1}\cdot H(1/4)+3}.$$

Now we proceed as in the proof of Proposition \ref{3diag}. We denote by $\gamma$ the cyclic permutation of the elements of $E$ defined by $(e_1,\dots,e_r)$.
In this case we have that $\omega_{2,\C}=\gamma^{k-1}$. Similarly, since the elements of $E$ are all integers equivalent to $1$ modulo $k-1$, $\omega_{1,\C}=\gamma^{-1}$. It follows that $\omega_{2,\C}\circ \omega_{1,\C}=\gamma^{k-2}$ which is a cyclic permutation on $E$ of order $r$ because $r$ is coprime with $k-2$.
Since we have assumed that $d=\gcd(n,k-1)=1$, both the conditions of Lemma \ref{Percorso2} are satisfied and hence $(\R,\C_E)$ is a solution of $P(A)$.
\end{proof}

\begin{rem}As already noted in Remark \ref{MigliorareEsponente}, also here, if $n$ is sufficiently large and given $\lambda<1/2$, we can replace the exponent $\frac{n}{k-1}\cdot H(1/4)$ of the previous proposition with $\frac{n}{k-1}\cdot H(\lambda)$. However, also in this case, we believe it is better to write the statement in the ``clearest'' case.
\end{rem}

\begin{prop}\label{k7}
Let $A$ be a cyclically $7$-diagonal Heffter array $\H_t(n;7)$ where $n> 120$ is an odd integer. Then the number of distinct solutions of $P(A)$ is, at least, of
$$ 4\binom{\lfloor n/6\rfloor}{ \lfloor n/24\rfloor}\gtrsim \frac{1}{\sqrt{n\pi}}2^{\lfloor\frac{n}{6}\rfloor\cdot H(1/4)+4}.$$
\end{prop}
\begin{proof}
We divide the proof in two cases. If $gcd(n,6)=1$, the thesis follows from Proposition \ref{power2}. In fact, in this case, the number of distinct solutions of $P(A)$ is, at least, of
$$ 4\binom{\lceil n/6\rceil}{ \lceil n/24\rceil}\geq  4\binom{\lfloor n/6\rfloor}{ \lfloor n/24\rfloor}.$$

Otherwise, we have that $gcd(n,6)=3$. In this case we consider a subset $E=(e_1,\dots,e_r)$
where $e_1<e_2<\dots<e_r$ of $[1,n]$ that satisfies the following properties:
\begin{itemize}
\item[(1)] $e_1=1$ and $e_2=2$;
\item[(2)] the elements $e_3,\dots,e_r$ of $E$ are integers equivalent to $3$ modulo $6$;
\item[(3)] $r=|E|$ is equivalent to $4$ modulo $5$.
\end{itemize}
We note that the number of integers equivalent to $3$ modulo $6$ in $[1,n]$ is $\lfloor\frac{n+3}{6}\rfloor\geq \lfloor\frac{n}{6}\rfloor$.
Now, we fix $r\equiv 4 \pmod{5}$ in $[\frac{n}{24}, \frac{n}{12}]$. Then the number of possible choices for a set $E$ of cardinality $r$ among the integers equivalent to $3$ modulo $6$ is, at least, of
$$ \binom{\lfloor n/6\rfloor}{ \lceil n/24\rceil}\geq \binom{\lfloor n/6\rfloor}{ \lfloor n/24\rfloor}\gtrsim \frac{1}{\sqrt{n\pi}}2^{\lfloor\frac{n}{6}\rfloor\cdot H(1/4)+2}.$$

As usual, we denote by $\gamma$ the cyclic permutation of the elements of $E$ defined by $(e_1,\dots,e_r)$. Here we have that $\omega_{2,\C}=\gamma^6$ and that, for $x\not\in \{e_1,e_2,e_3\}$, $\omega_{1,\C}=\gamma^{-1}$. It follows that, if $x\not\in \{e_1,e_2,e_3\}$, $\omega_{2,\C}\circ \omega_{1,\C}=\gamma^{5}$, that is $\omega_{2,\C}\circ \omega_{1,\C}(e_i)=e_{i+5}$ for $i\not \in \{1,2,3\}$ and where the indices are considered modulo $r$.
Due to the definition, we also have that $\omega_{1,\C}(e_1)=e_1$, $\omega_{1,\C}(e_2)=e_2$ and $\omega_{1,\C}(e_3)=e_r$.
It means that $\omega_{2,\C}\circ \omega_{1,\C}(e_1)=\gamma^{6}(e_1)=e_7$, $\omega_{2,\C}\circ \omega_{1,\C}(e_2)=\gamma^{6}(e_2)=e_8$ and $\omega_{2,\C}\circ \omega_{1,\C}(e_3)=e_6$.

Since $r\equiv 4 \pmod{5}$, we have that: $$\gamma^5=(e_1,e_6,e_{6+5},\dots, e_{r-3}, e_2,e_7,\dots, e_{r-2}, e_3,e_{3+5},\dots, e_{r-1}, e_4,\dots, e_r, e_5,\dots, e_{r-4}).$$
It follows that $\omega_{2,\C}\circ \omega_{1,\C}$ is the cycle of order $r$ given by:
$$(e_1,e_7,\dots, e_{r-2}, e_3,e_6,e_{6+5},\dots, e_{r-3}, e_2,e_{8},\dots, e_{r-1}, e_4,\dots, e_r, e_5,\dots, e_{r-4}).$$
Moreover, since $E$ covers all the congruence classes modulo $3$ in $[1,n]$, both the conditions of Lemma \ref{Percorso2} are satisfied and $(\R,\C_E)$ is a solution of $P(A)$. Finally, the thesis follows because, due to Lemmas \ref{-R-C} and \ref{CR}, from each such solution $(\R,\C_E)$ we obtain four different solutions of $P(A)$.
\end{proof}

\begin{thm}\label{DiagBi}
Let $n\geq 3$ and $t$ be such that either $t\in \{1,2\}$ and $n\equiv 1\pmod{4}$ or $t=3$ and $n\equiv 3\pmod{4}$ or $t\in \{n,2n\}$ and $n$ is odd. Then, set $v=6n+t$, the number of non-isomorphic simple $3$-gonal biembeddings of $K_{\frac{6n+t}{t}\times t}$ is, at least, of:
$$\frac{2^{n/2}}{9n^2}\approx\begin{cases} 2^{v/12+o(v)} \mbox{ if }t\in\{1,2,3\};\\
2^{v/14+o(v)} \mbox{ if }t=n;\\
2^{v/16+o(v)} \mbox{ if }t=2n.\end{cases}$$
\end{thm}
\proof
For these sets of parameters $n$ and $t$, Lemma \ref{ExistenceDiagonalGS} assures the existence of a cyclically $k$-diagonal $\H_t(n;3)$, say $A$. Then, due to Proposition \ref{3diag}, the problem $P(A)$ admits at least $2^{n/2+2}$ solutions. The thesis follows from Remark \ref{NonIsoSol}.
\endproof

\begin{rem}\label{confronto}
If $t=1$, namely if we are considering the complete graph $K_{6n+1}$, the lower bound of Theorem \ref{DiagBi} is surely worse than the ones already obtained in the literature (see \cite{Bonnington, GG08,GGS,GK10A,GK10B}). On the other hand, we want to underline that our result is still exponential in $v$ and holds also for other values of $t$.
\end{rem}
\begin{thm}\label{DiagBi2}
Let $k\in\{5,7,9\}$, let $n\geq 120$ and $t$ be such that either $t\in \{1,2\}$ and $nk\equiv 3 \pmod{4}$ or $t=k$ and $n\equiv 3 \pmod{4}$. Then, set $v=2nk+t$, the number of non-isomorphic simple $k$-gonal biembeddings of $K_{\frac{2nk+t}{t}\times t}$ is, at least, of:
$$ \frac{\binom{\lfloor n/(k-1)\rfloor}{ \lfloor n/(4k-4)\rfloor}}{(nk)^2}\gtrsim \frac{\sqrt{\frac{ 2(k-1)}{3n\pi}}2^{\lfloor\frac{n}{k-1}\rfloor\cdot H(1/4)+1}}{(nk)^2}\approx 2^{v\cdot \frac{H(1/4)}{2k(k-1)}+o(v,k)}.$$
\end{thm}
\proof
We proceed as in the proof of Theorem \ref{DiagBi} by using Propositions \ref{power2} and \ref{k7} instead of Proposition \ref{3diag}.
\endproof
\begin{thm}\label{DiagBi3}
Let $k>9$ be odd, let $n\geq 4k-3$ and $t$ be such that either $t\in \{1,2\}$ and $nk\equiv 3 \pmod{4}$ or $t=k$ and $n\equiv 3 \pmod{4}$. Assume also that $\gcd(n,k-1)=1$. Then, set $v=2nk+t$, the number of non-isomorphic, non necessarily simple, $k$-gonal biembeddings of $K_{\frac{2nk+t}{t}\times t}$ is, at least, of:
$$ \frac{\binom{\lceil n/(k-1)\rceil}{ \lceil n/(4k-4)\rceil}}{(nk)^2}\gtrsim \frac{\sqrt{\frac{ 2(k-1)}{3n\pi}}2^{\frac{n}{k-1}\cdot H(1/4)+1}}{(nk)^2}\approx 2^{v\cdot \frac{H(1/4)}{2k(k-1)}+o(v,k)}.$$
\end{thm}
\proof
We proceed as in the proof of Theorem \ref{DiagBi} by using Lemma \ref{ExistenceDiagonal} instead of Lemma \ref{ExistenceDiagonalGS} and Proposition \ref{power2} instead of Proposition \ref{3diag}.
\endproof

\section*{Acknowledgements}
The authors were partially supported by INdAM--GNSAGA.


\begin{thebibliography}{50}

\bibitem{A} D.S. Archdeacon,
\textit{Heffter arrays and biembedding graphs on surfaces},
Electron. J. Combin. \textbf{22} (2015) \#P1.74.

\bibitem{ADDY} D.S. Archdeacon, J.H. Dinitz, D.M. Donovan \and E.S. Yaz\i c\i,
\textit{Square integer Heffter arrays with empty cells},
Des. Codes Cryptogr. \textbf{77} (2015), 409--426.


\bibitem{Bonnington}C. P. Bonnington, M. J. Grannell, T. S. Griggs \and J. Siran, \textit{Exponential Families of Non-Isomorphic Triangulations
of Complete Graphs
}, J. Combin. Theory Ser. B. \textbf{78} (2000), 169--184.


\bibitem{BCDY} K. Burrage, N.J. Cavenagh, D. Donovan \and E.\c{S}. Yaz\i c\i,
\textit{Globally simple Heffter arrays $H(n;k)$ when $k\equiv 0,3 \pmod{4}$},
Discrete Math. \textbf{343} (2020), 111787.

\bibitem{CDDY} N.J. Cavenagh, J. Dinitz, D. Donovan \and E.S. Yaz\i c\i,
\textit{The existence of square non-integer Heffter arrays},
Ars Math. Contemp. \textbf{17} (2019), 369--395.


\bibitem{CDDYbiem} N.J. Cavenagh, D. Donovan \and E.\c{S}. Yaz\i c\i,
\textit{Biembeddings of cycle systems using integer Heffter arrays},
J. Combin. Des. \textbf{28} (2020), 900--922.

\bibitem{CDP} S. Costa, M. Dalai \and A. Pasotti,
\textit{A tour problem on a toroidal board},
Australas. J. Combin. \textbf{76} (2020), 183--207.

\bibitem{RelH} S. Costa, F. Morini, A. Pasotti \and M.A. Pellegrini,
\textit{A generalization of Heffter arrays},
J. Combin. Des. \textbf{28} (2020), 171--206.


\bibitem{CMPPHeffter} S. Costa, F. Morini, A. Pasotti \and M.A. Pellegrini,
\textit{Globally simple Heffter arrays and orthogonal cyclic cycle decompositions},
Austral. J. Combin. \textbf{72} (2018), 549--593.

\bibitem{CPPBiembeddings} S. Costa, A. Pasotti \and M.A. Pellegrini,
\textit{Relative Heffter arrays and biembeddings},
Ars Math. Contemp. \textbf{18} (2020), 241--271.

\bibitem{CPEJC} S. Costa \and A. Pasotti, \textit{On $\lambda$-fold relative Heffter arrays and biembedding multigraphs on surfaces},
Europ. J. Combin. \textbf{97} (2021), 103370.


\bibitem{DM} J.H. Dinitz \and A.R.W. Mattern,
\textit{Biembedding Steiner triple systems and $n$-cycle systems on orientable surfaces},
Australas. J. Combin. \textbf{67} (2017), 327--344.

\bibitem{DW} J.H. Dinitz \and I.M. Wanless,
\textit{The existence of square integer Heffter arrays},
Ars Math. Contemp. \textbf{13} (2017), 81--93.


\bibitem{GG} M. J. Grannell \and T. S. Griggs, \textit{Designs and topology}, Surveys in Combinatorics 2007, London Mathematical
Society Lecture Note Series, 346 (A. Hilton and J. Talbot, eds.), Cambridge University Press, Cambridge (2007), 121--174.

\bibitem{GG08} M. J. Grannell \and T. S. Griggs, \textit{A lower bound for the number of triangular embeddings of some complete graphs and complete regular tripartite graphs}, J. Combin. Theory Ser. B. \textbf{98} (2008), 637--650.



\bibitem{GGS} M. J. Grannell, T. S. Griggs \and J. Siran, \textit{Recursive constructions for triangulations}, J. Graph Theory. \textbf{39} (2002), 87--107.

\bibitem{GGSHamiltonian} M. J. Grannell, T. S. Griggs \and J. Siran, \textit{Hamiltonian embeddings from triangulation}, Bull. London Math. Soc. \textbf{39} (2007), 447--452.

\bibitem{GK10A} M. J. Grannell \and M. Knor, \textit{A lower bound for the number of orientable triangular embeddings of some complete graphs}, J. Combin. Theory Ser. B. \textbf{100} (2010), 216--225.

\bibitem{GK10B} M. J. Grannell \and M. Knor, \textit{On the number of Triangular Embeddings of Complete Graphs and Complete Tripartite Graphs},  J. Graph Theory. \textbf{69} (2012), 370--382.


\bibitem{GT} J.L. Gross \and T.W. Tucker,
\textit{Topological Graph Theory},
John Wiley, New York, 1987.

\bibitem{Korzhik10}V.P. Korzhik, \textit{Exponentially many nonisomorphic genus embeddings of $K_{n,m}$}, Discrete Math. \textbf{310} (2010), 2919--2924.

\bibitem{Korzhik12}V.P. Korzhik, \textit{Generating Nonisomorphic Quadrangular Embeddings of a Complete Graph},  J. Graph Theory. \textbf{74} (2013), 133--142.

\bibitem{Korzhik21}V.P. Korzhik, \textit{A simple construction of exponentially many nonisomorphic orientable triangular embeddings of $K_{12s}$},  Art Discrete Appl. Math. \textbf{4} (2021), P1.07.


\bibitem{Korzhik}V.P. Korzhik \and H.J. Voss, \textit{On the Number of Nonisomorphic Orientable Regular Embeddings of Complete Graphs}, J. Combin. Theory Ser. B. \textbf{81} (2001), 58--76.

\bibitem{Korzhik2}V.P. Korzhik \and H.J. Voss, \textit{Exponential Families of Non-isomorphic Non-triangular Orientable Genus Embeddings of Complete Graphs}, J. Combin. Theory Ser. B. \textbf{86} (2002), 186--211.


\bibitem{LNW} S. Lawrencenko, S. Negami, A. T. White, \textit{Three nonisomorphic triangulations of an orientable surface with the same complete graph}, Discrete Math. \textbf{135} (1994), 367--369.

\bibitem{Moh} B. Mohar,
\textit{Combinatorial local planarity and the width of graph embeddings},
Canad. J. Math. \textbf{44} (1992), 1272--1288.

\bibitem{MT} B. Mohar \and C. Thomassen,
\textit{Graphs on surfaces},
Johns Hopkins University Press, Baltimore, 2001.


\bibitem{MP} F. Morini \and M.A. Pellegrini,
\textit{On the existence of integer relative Heffter arrays},
Discrete Math. \textbf{343} (2020), 112088.

\bibitem{MP3} F. Morini \and M.A. Pellegrini,
\emph{Rectangular Heffter arrays: a reduction theorem}, preprint available at https://arxiv.org/abs/2107.08857v2.

\bibitem{R} G. Ringel, \textit{Map color theorem}, Spring‐Verlag Press, Berlin, 1974.
\bibitem{JS} J. Siran,
\textit{Graph Embeddings and Designs},
in: \textit{Handbook of Combinatorial Designs}. Edited by C. J. Colbourn and J. H. Dinitz. Second edition. Discrete
Mathematics and its Applications. Chapman \& Hall/CRC, Boca Raton, 2007.
\end{thebibliography}
\end{document}